\documentclass[12pt,a4paper]{article}

\usepackage[T1]{fontenc}
\usepackage[utf8]{inputenc}
\usepackage{float}
\usepackage{graphicx}
\usepackage{array}     
\usepackage{multirow}  
\usepackage{subcaption} 
\usepackage[thinlines]{easytable}
\usepackage{hyperref}
\usepackage{wrapfig}
\usepackage{dingbat}
\usepackage{color}
\usepackage{refcount}
\usepackage[table]{xcolor}

\usepackage[toc,page]{appendix}
\usepackage{bbm}
\usepackage{enumitem}
\usepackage{amsmath,amssymb,amsthm}
\usepackage{mathtools}
\usepackage[english]{babel}
\usepackage{natbib}
\usepackage{xcolor}
\usepackage[normalem]{ulem}
\usepackage{bibunits}
\usepackage{mathtools}
\usepackage[table]{xcolor}

\usepackage{threeparttable}
\usepackage{siunitx}
\usepackage{booktabs}
\usepackage{multirow}
\usepackage{adjustbox}

\mathtoolsset{showonlyrefs}

\theoremstyle{definition}
\theoremstyle{plain}
\newtheorem{theo}{Theorem}[section]
\newtheorem{lemma}[theo]{Lemma}

\newtheorem{cor}[theo]{Corollary}
\theoremstyle{definition}

\newtheorem{assumptions}[theo]{Assumptions}
\theoremstyle{remark}
\newtheorem{rem}[theo]{Remark}

\numberwithin{equation}{section}

\definecolor{violet}{RGB}{148,0,211}
\definecolor{olivegreen}{RGB}{107,142,35}

\newcommand{\cond}{\overset{d}{\to}}
\newcommand{\R}{\mathbb{R}}

\newcommand{\E}{\mathbb{E}}
\renewcommand{\Pr}{\mathbb{P}}
\newcommand{\conp}{\stackrel{\mathbb{P}}{\longrightarrow}}
\newcommand{\Cspace}{C([0,1])}
\newcommand{\norminf}[1]{\left\lVert #1 \right\rVert_{\infty}}
\newcommand{\floor}[1]{\left\lfloor #1 \right\rfloor}

\newcommand{\barX}{\overline X}

\usepackage{algorithm}
\usepackage{algpseudocode}

\begin{document}
\date{}
\author{Patrick Bastian \\
Department of Mathematics, Aarhus University \\
patrick.bastian@math.au.dk}

\title{Selfnormalization for relevant inference with supremum-type statistics}
\maketitle
\begin{abstract}
    We develop a selfnormalized approach to inference for relevant changes in functional time series measured by the supremum norm. The main difficulty is that the supremum norm is not Hadamard differentiable, so standard projection-based selfnormalization does not apply and the limiting distribution may depend on the geometry of the extremal set and the long-run covariance structure. We address this problem by replacing the supremum norm with a smooth log-sum-exp approximation and constructing a projected selfnormalizer from its derivative. The resulting statistic has an asymptotically pivotal distribution that is free of long-run covariance nuisance parameters and depends only on the break location. We derive explicit smoothing-bias expansions for both isolated nondegenerate extrema and extremal sets of positive measure. To avoid direct estimation of geometric quantities such as the number, curvature, or measure of the extrema, we combine several smoothing levels to cancel the leading bias terms. This yields an asymptotically exact test for relevant changes under mild regularity conditions. More generally, the proposed smoothing and bias-correction principles provide a framework for combining selfnormalization with supremum-type statistics in problems involving relevant hypotheses. 
\end{abstract}

\defaultbibliographystyle{apalike}
\defaultbibliography{reference}
\begin{bibunit}
\section{Introduction }

Statistical inference for dependent data typically requires estimation of a
long-run variance or long-run covariance operator. Such estimators involve
bandwidth, truncation, or block-length choices, and their finite-sample
performance can be sensitive to these tuning parameters. But even if an oracle choice for the parameters is provided the size of the resulting tests is often inflated when sample sizes are moderately small or autocorrelation strong. Selfnormalization (see \cite{shao:2010} and for a survey \cite{shao:2015}) provides an alternative: recursive estimates are used to construct a random normalizer so that the resulting statistic has an (almost) pivotal limiting distribution that does not depend on the long-run
covariance structure.\\

This principle is particularly attractive for change point problems,
where estimating the long-run covariance is complicated by the possible
presence and unknown location of a break. \cite{shao:zhang:2010} propose a selfnormalized CUSUM based change point test for univariate time series. Building on their insights several extensions for more complicated settings have been developed (e.g. \cite{wang:zhu:volgushev:shao:2022} for high dimensional or \cite{Chang:Ng:Yau:2021} for online change point detection).\\

In many modern applications data are (discretizations of) smooth curves rather than intrinsically finite dimensional objects. Accordingly a rich literature on change point tests for such data (e.g. \cite{berkes:gabrys:horvath:kokoszka:2009}, \cite{aston:kirch:2012}, \cite{aue:rice:sonmez:2018}) has developed, with some authors applying the principle of selfnormalization also in this context (e.g. \cite{zhang:shao:2015} and \cite{dette:kokot:volgushev:2020}). All of these papers have in common that their theoretical setting is the Hilbert space of square integrable functions and that their test procedures are based on $L^2$ type statistics. Yet in many cases it is more natural to execute the analysis in the Banach space $(C([0,1]),\|\cdot\|_\infty)$ as the supremum norm offers intuitive visual interpretation, improved detection of localized deviations and also provides simultaneous confidence bands for the parameter of interest. This perspective was first presented in \cite{dette:kokot:aue:2020} and has since been developed in a number of different directions (e.g. \cite{bastian:basu:dette:2024} for multiple and \cite{bastian:dette:2025b} for gradual and \cite{kutta:doernemann:2025} for online change point problems, \cite{cai:hu:2024} for sample scheme agnostic change point inference, \cite{Telschow:Davenport:Schwartzman:2022} for simultaneous confidence bands for a general class of parameters, \cite{Munko:Ditzhaus:Pauly:Smaga:2025} for confidence bands in functional MANOVA and \cite{Chang:McKeague:2022} for an approach that is able to cope with discontinuities).\\

The absence of selfnormalization results in this framework is no accident. The classical selfnormalization approach is driven by an asymptotically linear expansion of smooth functionals of the CUSUM process. This expansion is then used to establish that the limiting distribution is dependent on a univariate long-run variance in a multiplicative way, allowing for cancellation by an appropriate normalizer. In contrast the map $f \mapsto \|f\|_\infty$ is not Hadamard differentiable. It possesses only directional differentials resulting in a non-linear derivative (i.e. directional Hadamard differentiability). As a result the limiting distributions often depend in a sensitive way on the geometry of certain extremal sets (see \cite{carcamo:cuevas:rodriguez:2020} for an in-depth analysis) and thus also on (potentially) the whole long-run covariance operator of the error processes and not just on the long-run variance of some functional. This prevents the standard projection-based selfnormalization construction from working. \\

The purpose of the present paper is to close this gap. For a functional time series $X_1,...,X_n \in C([0,1])$ that satisfies
\[
    X_i(t)=\begin{cases}
        \mu_1(t)+\epsilon_i(t), \ 1 \leq i \leq k_0\\
        \mu_2(t)+\epsilon_i(t), \ k_0 < i \leq n
    \end{cases}
\]
we consider the hypotheses 
\begin{align}
    H_0(\Delta):M\leq \Delta, \quad \text{versus} \quad H_1(\Delta):M>\Delta
\end{align}
where $M:=\|\mu_1-\mu_2\|_\infty$. For a thorough motivation and introduction to this type of hypotheses we refer to \cite{dette:kokot:aue:2020}. Our goal is to derive a selfnormalized testing procedure for this problem. To that end we consider a smoothed maximum functional given by
\[
  \Psi_\beta(f)
  =\frac{1}{\beta}\log\int_0^1
    \{\exp(\beta f(t))+\exp(-\beta f(t))\}\,dt.
\]
which (in a discretized form) has also been used in \cite{chernozhukov:chetverikov:kato:2013} to derive high dimensional Gaussian approximations. When $\beta\to \infty$ this functional converges to the supremum norm of $f$ and for fixed $\beta$ it is a Fréchet differentiable function whose curvature is proportional to $\beta$. This is where the key tension of the construction we propose lies: Let $d(\cdot)=\mu_1(\cdot)-\mu_2(\cdot)$ and note that larger $\beta$ reduces the bias
\[
    b_\beta(d)=\Psi_\beta(d)-\lVert d\rVert_\infty,
\]
but incurs a larger Taylor remainder. As a result of this, naive centering at $M=\|d\|_\infty=\|\mu_1-\mu_2\|_\infty$ fails and a bias correction becomes necessary. Additional difficulties arise as explicit bias expansions are extremely sensitive to the geometry of $\mathcal E(d)=\{ t | |d(t)|=M\}$, even yielding rates of different orders in common scenarios. To resolve this issue we develop a procedure in the spirit of \cite{Schucany:Sommers:1977} that cancels the bias for a wide variety  of geometric structures of $\mathcal{E}$, leading to the first selfnormalized testing procedure in the  $C([0,1])$ functional data framework. 

We close this section with a summary of our main contributions
\begin{enumerate}[label=(\roman*)]
  \item We develop direct selfnormalized inference for relevant changes
  measured in the supremum norm. This yields a limiting distribution that is free of long-run
  covariance nuisance parameters and that only depends on the break
  location. We stress that the proposed methodology can easily be extended to other relevant testing problems involving supremal quantities. 
  \item We derive explicit smoothing-bias expansions for qualitatively different
  extremal geometries, including finitely many nondegenerate extrema and flat
  extremal sets. The expansions clarify how the smoothing parameter interacts
  with the local geometry of the mean contrast and with the stochastic
  linearization error.

  \item We construct a bias-corrected selfnormalized statistic by combining
  several smoothing levels. The correction removes the leading bias terms
  across a broad class of extremal configurations without requiring direct
  estimation of the number, curvature, signs, or measure of the extrema. 
\end{enumerate}

\section{Methodology}

Consider a functional time series
\[
        X_1,\ldots,X_n \in \Cspace,
\]
where \(\Cspace\) is equipped with the supremum norm \(\norminf{f}=\sup_{t\in[0,1]}|f(t)|\).  Throughout, the index \(i\) denotes time, while \(t\in[0,1]\) denotes the argument of the observed function.  The basic model is
\[
        X_i(t)=\mu_i(t)+\epsilon_i(t),\qquad t\in[0,1],
\]
where \(\mu_i\in\Cspace\) is the mean function at time \(i\) and \((\epsilon_i)\) is a centered dependent functional error process. There is an unknown break fraction \(\tau_0\in(0,1)\), with \(k_0=\floor{n\tau_0}\), such that
\[
        \mu_i=\mu_1 \quad (i\le k_0),
        \qquad
        \mu_i=\mu_2 \quad (i>k_0),
\]
for two functions \(\mu_1,\mu_2\in\Cspace\).  The size of the change is measured by
\[
        d=\mu_1-\mu_2,
        \qquad
        M=\norminf{d}.
\]
and we want to test the hypotheses 
\begin{align}
\label{eq:hypotheses-defin}
   H_0: M\le \Delta
        \qquad\text{versus}\qquad
        H_1: M>\Delta, 
\end{align}
for some $\Delta>0$ that distinguishes practically meaningful from statistically detectable changes.\\

The standard starting point is the functional CUSUM process
\[
        \widehat U_n(s,t)
        =\frac1n\sum_{i=1}^{ns}\{X_i(t)-\barX_n(t)\},
        \qquad
        \barX_n=\frac1n\sum_{i=1}^n X_i .
\]
where sums with non-integer boundaries are to be understood as linearly interpolated. A classical supremum-based scan statistic for the presence of a change is
\[
        \widehat{\mathcal C}_n
        =\sup_{s\in[0,1]}
          \norminf{\widehat U_n(s,\cdot)},
\]
A corresponding estimate of the break fraction is
\[
        \widehat\tau_n\in
        \operatorname*{arg\,max}_{s\in[1/n,1-1/n]}
        \norminf{\widehat U_n(s,\cdot)},
\]
where the restriction in the interval is only present to ensure that the selfnormalizer we define later is always well-defined. Writing \(\widehat k_n=\floor{n\widehat\tau_n}\), the post-estimation contrast between the two regimes is
\[
        \widehat d_n(t)
        =\frac1{\widehat k_n}\sum_{i=1}^{\widehat k_n}X_i(t)
        -\frac1{n-\widehat k_n}\sum_{i=\widehat k_n+1}^{n}X_i(t),
        \qquad t\in[0,1],
\]
and the direct supremum norm estimator of the relevant change size is \(\widehat M_n=\norminf{\widehat d_n}\). Our goal is selfnormalized inference for $M$ using this estimate. As documented in \cite{dette:kokot:aue:2020} the statistical difficulty is that the map \(f\mapsto\norminf{f}\) is not a smooth linear functional.  Its limiting distribution is non-Gaussian and depends on the geometry of the extremal set of \(d\), making the classical approach to selfnormalization as in \cite{shao:zhang:2010} unsuitable.\\

To avoid this non-smoothness, we replace the supremum norm by a
differentiable soft maximum. For \(\beta>0\), define
\[
        \Psi_\beta(f)
        =
        \frac1\beta
        \log\int_0^1
        \left\{\exp(\beta f(t))+\exp(-\beta f(t))\right\}\,dt,
        \qquad f\in\Cspace .
\]
The selfnormalized statistic is built after the split \(\widehat k_n\)
has been selected. To define the selfnormalizer, let
\[
        S_{a,b}(t)=\sum_{i=a}^{b}X_i(t),
        \qquad 1\le a\le b\le n .
\]
At the estimated split \(\widehat k_n\), introduce the left and right
bridge processes
\[
        B_{q}^{L}(s)
        =
        S_{1,q}(s)
        -
        \frac{q}{\widehat k_n}S_{1,\widehat k_n}(s),
        \qquad 1\le q\le \widehat k_n,
\]
and
\[
        B_{q}^{R}(s)
        =
        S_{q,n}(s)
        -
        \frac{n-q+1}{n-\widehat k_n}
        S_{\widehat k_n+1,n}(s),
        \qquad \widehat k_n+1\le q\le n .
\]
The projected change point selfnormalizer is
\begin{align}
\label{eq:defin-Wn}
    \widehat W_n^{(\beta)}
        =
        \frac1{n^2}
        \left[
        \sum_{q=1}^{\widehat k_n}
        \left\{
        D\Psi_{\beta,\widehat d_n}
        \left(B_q^{L}\right)
        \right\}^2
        +
        \sum_{q=\widehat k_n+1}^{n}
        \left\{
        D\Psi_{\beta,\widehat d_n}
        \left(B_q^{R}\right)
        \right\}^2
        \right],
\end{align}
where $D\Psi_{\beta,f}(h)$ is the derivative of $\Psi_\beta$ at $f$, evaluated in $h$, i.e.
\[
    D\Psi_{\beta,f}(h)
=
\frac{
\int_0^1 h(t)\{e^{\beta f(t)}-e^{-\beta f(t)}\}\,dt
}{
\int_0^1 \{e^{\beta f(t)}+e^{-\beta f(t)}\}\,dt
}.
\]

A natural statistic for the relevant hypotheses then is
\[
        T_n^{\mathrm{SN}}(\Delta)
        =
        \frac{
        \sqrt n\{\Psi_{\beta_n}(\widehat d_n)-\Delta\}
        }{
        \sqrt{\widehat V_n^{(\beta_n)}}
        }, \qquad 
        \widehat V_n^{(\beta)}
        =
        \left\{
        \frac{n^2}{\widehat k_n(n-\widehat k_n)}
        \right\}
        \widehat W_n^{(\beta)} .
\]
where \(\beta_n\to\infty\) is allowed to increase with the sample size. The role of \(\beta_n\) is to balance stochastic linearization against deterministic smoothing bias.  A Taylor expansion around the population contrast \(d\) gives
\begin{align}
\label{eq:taylor-expansion}
     \Psi_{\beta_n}(\widehat d_n)-\Psi_{\beta_n}(d)
        =D\Psi_{\beta_n,d}(\widehat d_n-d)+R_{n,\beta_n},
\end{align}
  
and the curvature of \(\Psi_\beta\) is of order \(\beta\).  Thus, when \(\norminf{\widehat d_n-d}=O_p(n^{-1/2})\), the $\sqrt{n}$-scaled remainder is controlled by the requirement
\[
        \beta_n=o(\sqrt n).
\]
Consequently, the smoothed statistic behaves like a selfnormalized scalar partial-sum statistic under the following assumptions:

\begin{assumptions}\label{ass:population-centering}
~~
\begin{enumerate}[label=\textnormal{(A\arabic*)},leftmargin=3.2em]
\item\label{ass:wip}
The linearly interpolated partial-sum process
\[
\mathbb U_n(r)
=
\begin{cases}
    \frac1{\sqrt n}\sum_{i=1}^{\floor{nr}}\epsilon_i,
\qquad &nr\in \mathbb N\\
\text{linearly interpoalted,}& nr \notin \mathbb N
\end{cases},
\]
satisfies the weak invariance principle
\[
\mathbb U_n\cond\mathbb B
\quad\text{in }C\bigl([0,1];\Cspace\bigr),
\]
where \(\mathbb B\) is a centered \(\Cspace\)-valued Brownian motion with
continuous sample paths. In particular, \(\E\norminf{\mathbb B(1)}^2<\infty\).

\item\label{ass:break-estimator}
If $M>0$ the estimated split obeys
\[
|\widehat k_n-k_0|=o_p(\sqrt n).
\]

\item\label{ass:beta-rate}
The smoothing sequence satisfies
\[
\beta_n\longrightarrow\infty,
\qquad
\beta_n=o(\sqrt n).
\]

\item\label{ass:nondegenerate}
Define
\[
\ell_n:=D\Psi_{\beta_n,d}\in\Cspace^*,
\qquad
\sigma_n^2:=\E\{\ell_n(\mathbb B(1))^2\},
\]
where $B$ is the Brownian motion from (A1). If $d\neq 0$ we require that the projected long-run variance is uniformly non-degenerate:
\[
\liminf_{n\to\infty}\sigma_n^2>0.
\]
\end{enumerate}
\end{assumptions}
Before stating a first weak convergence result let us briefly discuss these assumptions. (A1) is satisfied under any of the widely used weak dependence concepts (i.e. $\alpha$ or $\beta$-mixing, $L^p$-$m$-approximability or physical dependence) whenever the data satisfy a smoothness condition of the form
\[
    \E[|X_i(t)-X_i(t')|^p]^{1/p}\leq C|t-t'|^\gamma
\]
where $\gamma$ and the decay of the dependence coefficients determine the minimal $p$ that is required. (A2) follows under the same conditions using, for instance, the results from \cite{hariz:wylie:zhang:2007}. (A3) gives an upper bound on the smoothing parameter $\beta$ and determines the leading order of the bias $\Psi_\beta(d)-\|d\|_\infty$. (A4) is necessary for the limit of the selfnormalizer to be well defined. It can be substantially weakened (to decay to 0 at certain rates) if one is willing to assume a strong approximation of the partial sum process (i.e. such as in \cite{dehling:1983}) instead of (A1).

\begin{theo}[Selfnormalized limit with population smoothing centering]
\label{theo:population-centered-sn}
Suppose Assumptions~\ref{ass:population-centering} hold and that $M>0$. Then
\[
T_n^{\mathrm{SN}}\bigl(\Psi_{\beta_n}(d)\bigr)
=
\frac{
\sqrt n\{\Psi_{\beta_n}(\widehat d_n)-\Psi_{\beta_n}(d)\}
}{
\sqrt{\widehat V_n^{(\beta_n)}}
}
\cond
\mathcal T_{\tau_0},
\]
where 
\[
    \mathcal T_{\tau_0}
    \overset{d}=
    \frac{Z}
    {\left\{
    \tau_0^2\int_0^1\mathbb B_1^0(u)^2\,du
    +(1-\tau_0)^2\int_0^1\mathbb B_2^0(u)^2\,du
    \right\}^{1/2}}=:\frac{Z}{D},
\] 
with $Z\sim N(0,1)$, $\mathbb B_i^0, \ i=1,2$ being standard Brownian bridges and all three being mutually independent. In particular, the limit depends on
the data-generating process only through the break fraction \(\tau_0\), and is free of long-run covariance nuisance parameters. Moreover,
\(\widehat V_n^{(\beta_n)}>0\) with probability tending to one.
\end{theo}

Centering the statistic at the true supremum norm \(M\) introduces the deterministic bias
\[
        b_{\beta_n}(d)=\Psi_{\beta_n}(d)-\norminf{d},
\]
which needs to be accounted for. In the following we present two ways of doing this, a generally valid but conservative one and a more restrictive but exact procedure.

\paragraph{Bounding the bias}
Using that $\exp(\beta f(t))=\exp(\beta M)\exp(\beta(f(t)-M))$ we have
\[
\int_0^1\exp(\beta f(t))dt=\exp(\beta M)\int_0^1\exp(\beta(f(t)-M))dt
\]
(and similarly for the negative exponential), taking logarithms it easily follows that 

\begin{align}
\label{eq:explicit-bias}
b_{\beta_n}(d)
&=
\beta_n^{-1}
\log\left(
\sum_{s\in\{-1,1\}}
\int_0^1
\exp\bigl(\beta_n\{s d(t)-M\}\bigr)\,dt
\right)
\nonumber\\
&\leq
\beta_n^{-1}
\log\left(1+\exp(-2\beta_n M)\right).
\end{align}
Under \(H_0(\Delta)\), the function
\[
M\mapsto
M-\Delta+
\beta_n^{-1}\log\left(1+\exp(-2\beta_n M)\right)
\]
is nondecreasing on \([0,\Delta]\). Consequently,
\[
\Psi_{\beta_n}(d)-\Delta
\leq
\beta_n^{-1}
\log\left(1+\exp(-2\beta_n\Delta)\right).
\]
One may therefore use the statistic
\begin{align}
\label{eq:conservative-bias-correction}
T_{n,1}^{\mathrm{SN}}(\Delta)
=
\frac{
\sqrt{n}\left\{
\Psi_{\beta_n}(\widehat d_n)
-\Delta
-\beta_n^{-1}
\log\left(1+\exp(-2\beta_n\Delta)\right)
\right\}
}{
\sqrt{\widehat V_n^{(\beta_n)}}
}
\end{align}
to test \(H_0(\Delta)\). The following
theorem is a direct corollary of the proof of
Theorem~\ref{theo:population-centered-sn}.

\begin{theo}
\label{theo:conservative-bias-correction}
Suppose that Assumptions~\ref{ass:population-centering} hold with
\(d=d_n\), where
\(M_n:=\lVert d_n\rVert_\infty\) satisfies
\(0<\liminf_{n\to\infty}M_n\leq \lim \sup_{n \to \infty}M_n<\infty\). Let \(q_{1-\alpha_0}(\tau)\) denote the
\((1-\alpha_0)\)-quantile of \(\mathcal T_\tau\). Define
\[
\gamma_n
:=
\frac{\sqrt{n\tau_0(1-\tau_0)}}{\sigma_n}
\left\{
M_n-\Delta
+b_{\beta_n}(d_n)
-\frac{
\log\left(1+\exp(-2\beta_n\Delta)\right)
}{\beta_n}
\right\}.
\]
If \(\gamma_n\to\gamma\in\overline{\mathbb R}\), then
\[
\mathbb P\left(
T_{n,1}^{\mathrm{SN}}(\Delta)
>
q_{1-\alpha_0}(\widehat\tau_n)
\right)
\longrightarrow
\mathbb P\left(
\frac{Z+\gamma}{H_{\tau_0}}
>
q_{1-\alpha_0}(\tau_0)
\right),
\]
where
\[
H_{\tau}
:=
\left\{
\tau^2\int_0^1\bigl(B_1^0(u)\bigr)^2\,du
+
(1-\tau)^2\int_0^1\bigl(B_2^0(u)\bigr)^2\,du
\right\}^{1/2},
\]
and \(Z,B_1^0,B_2^0\) are mutually independent, with
\(Z\sim N(0,1)\) and \(B_1^0,B_2^0\) standard Brownian bridges.
For \(\gamma\in\mathbb R\), the limiting rejection probability equals
\[
\mathbb E\left[
1-\Phi\left\{
q_{1-\alpha_0}(\tau_0)H_{\tau_0}-\gamma
\right\}
\right].
\]

In particular, under \(H_0(\Delta)\),
\begin{align}
\label{eq:bias-bound}
   M_n-\Delta+b_{\beta_n}(d_n)
\leq
\frac{
\log\left(1+\exp(-2\beta_n\Delta)\right)
}{\beta_n}, 
\end{align}
so the test has asymptotic level at most \(\alpha_0\). Its rejection
probability converges to \(0\) if \(\gamma_n\to-\infty\), and to \(1\)
if \(\gamma_n\to+\infty\).
\end{theo}
Unfortunately the bound \eqref{eq:bias-bound} is an equality only when $d(\cdot)$ is constant. This typically leads to a loss of power against non-constant (i.e. $d(\cdot)$ not being a constant function) local alternatives of (for instance) size $n^{-1/2}$ as $\gamma$ may be equal to $-\infty$ in this case. In the next subsections we develop more explicit bias expansions that will allow us to solve this problem. We close the subsection with a small technical remark: Theorem \ref{theo:conservative-bias-correction} implicitly makes use of the continuity $q_{1-\alpha_0}(\cdot)$, which follows from the continuity of the cdf of $Z$ and a simple conditioning argument and the dominated convergence theorem.

\paragraph{An explicit bias expansion: finitely many extrema}
Finding a statistically fruitful expansion for
\[
\log\Big(\sum_{s \in \{-1,1\}}\int_0^1\exp(\beta_n(sf(t)-M))\Big)
\]
into estimable parameters is a non-trivial task. We will use Laplace's method which provides expansions for integrals of the form
\[
    \int_a^b\exp(\beta f(t))dt
\]
provided that $\beta$ is large and $f$ sufficiently regular. The method has a long history in statistics, starting with the original paper by \cite{Laplace:1986}. It is used in various contexts such as model selection  \cite{schwarz:1978}, posterior moment approximation \cite{Tierney:kadane:1986} and in classification via Gaussian processes \cite{Williams:Barber:1998} to name just a few.

The key idea is that near a global extremum (say a maximum) $x_0$ of $f$ we may write 
\[
 f(x)=f(x_0)-\frac{|f''(x_0)|(x-x_0)^2}{2}+O((x-x_0)^3)
\]
where we use that the second derivative has to be non-positive as $x_0$ is a maximal point. Consequently
\[
    \int_{x_0-\delta}^{x_0+\delta}\exp(\beta f(t))dt\simeq \exp(\beta f(x_0))\int_{x_0-\delta}^{x_0+\delta}\exp\Big(-\frac{\beta|f''(x_0)|(t-x_0)^2}{2}\Big)dt.
\]
For large $\beta$ this integral is essentially a Gaussian integral over the whole real line, yielding
\[
    \int_{x_0-\delta}^{x_0+\delta}\exp(\beta f(t))dt\simeq \sqrt{\frac{2\pi}{\beta |f''(x_0)|}}\exp(\beta f(x_0)).
\]
Consequently, if the set of extremal points of the contrast function $d$ is finite, contained in the interior of [0,1] and contains only locally quadratic maxima we may use a more quantitative version of this argument to obtain the following Theorem 

\begin{theo}[Laplace expansion of the soft-maximum bias]
\label{theo:laplace-bias}
Let \(d\in C^4([0,1])\), $M=\|d\|_\infty>0$ and suppose that  
\[
\mathcal E(d)
=
\left\{
t\in[0,1]:|d(t)|=M
\right\}
=
\{t_1,\ldots,t_m\}
\subset(0,1),
\]
is a finite set. For \(j=1,\ldots,m\), define
\[
s_j=\operatorname{sgn}\{d(t_j)\}\in\{-1,1\},
\]
and suppose that each extremal point is nondegenerate, that is,
\[
d'(t_j)=0,
\qquad
\lambda_j:=-s_jd''(t_j)>0,\qquad  j=1,...,m.
\]
Furthermore, let
\[
\gamma_j=s_jd'''(t_j),
\qquad
\kappa_j=s_jd''''(t_j),
\]
and define
\[
S_0(d)
=
\sum_{j=1}^m\lambda_j^{-1/2}, c_j
=
\frac{\kappa_j}{8\lambda_j^2}
+
\frac{5\gamma_j^2}{24\lambda_j^3},
\qquad
S_1(d)
=
\sum_{j=1}^m\lambda_j^{-1/2}c_j.
\]
Then, as \(\beta\to\infty\),
\[
\int_0^1
\left\{
e^{\beta d(t)}+e^{-\beta d(t)}
\right\}\,dt
=
e^{\beta M}
\sqrt{\frac{2\pi}{\beta}}
S_0(d)
\left[
1+
\frac{S_1(d)}{\beta S_0(d)}
+
o\left(\frac1\beta\right)
\right].
\]
Consequently,
\begin{align}
\label{eq:discrete-expansion}
    \Psi_\beta(d)
    =
    M+
    \frac{
    -\frac12\log\beta
    +\frac12\log(2\pi)
    +\log S_0(d)
    }{\beta}
    +
    \frac{S_1(d)}{S_0(d)\beta^2}
    +
    o(\beta^{-2}).
\end{align}
In particular, the smoothing bias
\[
b_\beta(d)=\Psi_\beta(d)-M
\]
satisfies
\[
b_\beta(d)
=
-\frac{\log\beta}{2\beta}
+
\frac{
\frac12\log(2\pi)+\log S_0(d)
}{\beta}
+
O(\beta^{-2}).
\]
\end{theo}

\subsection{An explicit bias expansion: flat extrema}

In some cases the extrema may be plateaus and not isolated points. A similar idea as in the previous subsection will be fruitful. We define 
\[
    g(t)=M-|d(t)|,
\]
as well as the extreme-volume
\[
    m=\lambda(\mathcal E(d)).
\]
Then, using $\exp(\beta d(t))+\exp(-\beta d(t))=\exp(\beta|d(t)|)+\exp(-\beta|d(t))$, we have by similar (but easier) arguments as in the beginning of the previous subsection that
\[
    \Psi_\beta(d)=M+\frac{\log(m)}{\beta}+o(\beta^{-1}).
\]
For a fruitful correction we will need a more precise form of the $o(\beta^{-1})$ term. To that end we define the near-extreme-volume function
\[
    F(u)=\lambda\{t| \ g(t)\leq u\}, \quad  u \geq 0,
\]
and specify the expansion in terms of its behaviour as $u \to 0$.

\begin{theo}[Bias expansion for an extremal set of positive measure]
\label{theo:plateau-bias}
Let \(d\in C([0,1])\) and suppose that the extremal set
\[
\mathcal E(d)
=
\{t\in[0,1]:|d(t)|=M\}
\]
has positive Lebesgue measure
\[
m=\lambda\{\mathcal E(d)\}>0.
\]
Further suppose that, for some constants \(K>0\) and \(\alpha>0\),
\[
F(u)
=
m+Ku^\alpha+o(u^\alpha),
\qquad u\downarrow0.
\]
Then, as \(\beta\to\infty\),
\[
\int_0^1
\left\{
e^{\beta d(t)}+e^{-\beta d(t)}
\right\}\,dt
=
e^{\beta M}
\left[
m
+
K\Gamma(1+\alpha)\beta^{-\alpha}
+
o(\beta^{-\alpha})
\right].
\]
Consequently,
\begin{align}
\label{eq:smooth-expansion}
    \Psi_\beta(d)
    =
    M
    +
    \frac{\log m}{\beta}
    +
    \frac{K\Gamma(1+\alpha)}{m}
    \beta^{-(1+\alpha)}
    +
    o\left(\beta^{-(1+\alpha)}\right).
\end{align}

In particular, the smoothing bias satisfies
\[
b_\beta(d)
=
\Psi_\beta(d)-M
=
\frac{\log m}{\beta}
+
\frac{K\Gamma(1+\alpha)}{m}
\beta^{-(1+\alpha)}
+
o\left(\beta^{-(1+\alpha)}\right).
\]
\end{theo}

\begin{rem}[The isolated-extrema case as a special case]
\label{rem:isolated-as-volume-case}
The near-extreme-volume formulation also contains the case of
isolated extremal points, although in that case
\[
m=\lambda\{\mathcal E(d)\}=0
\]
and the expansion in Theorem~\ref{theo:plateau-bias} cannot be used
by simply setting \(m=0\).

More generally, suppose that
\[
F(u)
=
Ku^\alpha+o(u^\alpha),
\qquad u\downarrow0,
\]
for some \(K>0\) and \(\alpha>0\). The same argument as in the proof
of Theorem~\ref{theo:plateau-bias} gives
\[
\int_0^1
\left\{
e^{\beta d(t)}+e^{-\beta d(t)}
\right\}\,dt
=
e^{\beta M}
K\Gamma(1+\alpha)\beta^{-\alpha}
\{1+o(1)\},
\]
and therefore
\[
\Psi_\beta(d)
=
M+
\frac{
-\alpha\log\beta
+
\log\{K\Gamma(1+\alpha)\}
}{\beta}
+
o(\beta^{-1}).
\]

In particular, under the assumptions of
Theorem~\ref{theo:laplace-bias}, a quadratic expansion around the
isolated extremal points gives
\[
F(u)
=
2\sqrt{2}\,S_0(d)\,u^{1/2}
+
o(u^{1/2}).
\]
Thus
\[
\alpha=\frac12,
\qquad
K=2\sqrt{2}\,S_0(d),
\]
and, since
\[
\Gamma\left(\frac32\right)=\frac{\sqrt\pi}{2},
\]
we obtain
\[
K\Gamma\left(\frac32\right)
=
\sqrt{2\pi}\,S_0(d).
\]
Consequently,
\[
\Psi_\beta(d)
=
M+
\frac{
-\frac12\log\beta
+\frac12\log(2\pi)
+\log S_0(d)
}{\beta}
+
o(\beta^{-1}),
\]
which is precisely the first-order part of
Theorem~\ref{theo:laplace-bias}. Hence the isolated-extrema and
positive-measure cases are both governed by the local behavior of
the near-extreme-volume function \(F\) at zero.
\end{rem}

\subsection{Correcting the bias}
Returning to the bias correction we observe that for sufficiently quickly growing $\beta_n$ we hence only need to concern ourselves with the term
\[
    -\frac{\log\beta_n}{2\beta_n}
+
\frac{
\frac12\log(2\pi)+\log S_0(d)
}{\beta_n}
\]
in the finitely many extrema case, and with the term
\[
    \log(m)\beta^{-1}
\]
in the plateau case. This is because the remainders are $O(\beta_n^{-2})$ and $O(\beta_n^{-(1+\alpha)})$, respectively, so that upon choosing $\beta_n$ moderately large they are $o(n^{-1/2})$. Estimating the constants $S_0(d)$ and $m$ is hairy business. For instance, focusing on $S_0(d)$, we need to estimate the number of extremal points, their sign as well as the second derivative at these points. Estimating the second derivative of  $d \in C^4[0,1]$ at rate $n^{-1/2}$ is possible when the observations are available on a sufficiently regular grid with at least (up to log factors) $n^{3/4}$ points (see \cite{Berger:Holzmann:2026}). In practice this approach is somewhat less appetizing due to two reasons: 1) Estimated change points near the beginning or end of the time series drastically reduce the effective sample size used to estimate $d$ and its derivative, 2) the estimation involves additional tuning parameters. Trying to estimate $\log(m)$ is no easier and differentiating between the two settings in a statistically sound way that is not dependent on a number of hyperparameters seems exceptionally difficult. \\

We therefore opt to instead employ a trick that is often used in the context of kernel density estimation (see \cite{Schucany:Sommers:1977}): We choose different sequences of the parameter $\beta_n$ and take a linear combination of the corresponding estimators $\Psi_{\beta_n}(\hat d_n)$ that preserves the zeroth order but cancels the first order bias terms. While the expansions \eqref{eq:discrete-expansion} and \eqref{eq:smooth-expansion} are structurally different it is nonetheless possible to cancel the biases with a three-term linear combination. To be precise we define
\[
    \mathcal L_{\beta_n}(\cdot, \mathbf a, \mathbf c)=a_1\Psi_{c_1\beta_n}(\cdot)+a_2\Psi_{c_2\beta_n}(\cdot)+a_3\Psi_{c_3\beta_n}(\cdot).
\]
for any $\mathbf a=(a_1,a_2,a_3) \in \R^3$ and $\mathbf c=(c_1,c_2,c_3) \in \R^3_+$. Similarly we define the associated selfnormalizer $\widehat{\mathcal V}_n^{\beta}$ by replacing the derivative in \eqref{eq:defin-Wn} by the derivative of $\mathcal{L}_{\beta_n}$, i.e.
\begin{align}
        \widehat{\mathcal{W}}_n^{(\beta)} 
        &=
        \frac1{n^2}
        \left[
        \sum_{q=1}^{\widehat k_n}
        \left\{
        D\mathcal L_{\beta,\widehat d_n}
        \left(B_q^{L}\right)
        \right\}^2
        +
        \sum_{q=\widehat k_n+1}^{n}
        \left\{
        D\mathcal L_{\beta,\widehat d_n}
        \left(B_q^{R}\right)
        \right\}^2
        \right],\\ 
        \widehat{\mathcal V}_n^{\beta}
        &=
        \left\{
        \frac{n^2}{\widehat k_n(n-\widehat k_n)}
        \right\}
        \widehat{\mathcal{W}}_n^{(\beta)} .
\end{align}
We want to choose these constants in such a way that the first order bias cancels while keeping the remainder in \eqref{eq:taylor-expansion} as small as possible. By the results in Lemma \ref{lem:softmax-differentiability} the upper bound on this remainder is proportional to $\beta_n\sum_{i=1}^3|a_i|c_i$ and we will hence try to minimize this objective over $(\mathbf a,\mathbf c)\in \R^3\times\R^3_+$ under the constraint that the first order bias vanishes. Because we require $\beta_n\to \infty$ we impose the additional constraint that $\min c_i\geq 1$, otherwise homogeneity would allow picking $c_i$ arbitrarily close to 0, thereby blowing up the  higher order bias terms. 

To ensure that the zeroth order bias is unchanged while the first order bias cancels some simple algebra yields the conditions
\[
    \sum _i a_i=1,\quad \sum_i\frac{a_i}{c_i}=0, \quad \sum_i \frac{a_i\log c_i}{c_i}=0, \ c_i\geq 1.
\]
Writing $\mathbf c=(1,x,y), \ 1<x <y$ we obtain a linear equation in $\mathbf a$.  Solving it gives
\[
    a_1=\frac{\log(x/y)}{D(x,y)},\quad a_2=\frac{x\log(y)}{D(x,y)}, \quad a_3=-\frac{y\log(x)}{D(x,y)}, \quad D(x,y)=(1-y)\log x+(x-1)\log y.
\]
From here some numerical optimization yields  that the minimizers $\mathbf a_0$ and $\mathbf{c}_0$ are approximately given by
\[
    \mathbf a_0 \approx (1.2509,-2.2555, 2.0046), \quad \mathbf c_0\approx(1,1.4475,6.5232).
\]
By the imposed constraints this then yields
\[
    \mathcal L_{\beta_n}(d):=\mathcal L_{\beta_n}(d,\mathbf a_0, \mathbf c_0)=M+o(n^{-1/2})
\]
in either case provided that $\beta_n$ is sufficiently large. The same arguments yielding Theorem \ref{theo:population-centered-sn} then also yield
\begin{theo}[Selfnormalized limit with bias correction]
\label{theo:bias-corrected-sn}
Suppose that Assumptions~\ref{ass:population-centering} hold. Further suppose that for $M>0$ one of the following two conditions holds:
\begin{itemize}
    \item[(1)] $\sqrt{n}\beta_n^{-2}=o(1)$ and the assumptions of Theorem \ref{theo:laplace-bias} hold
    \item[(2)] $\sqrt{n}\beta_n^{-(1+\alpha)}=o(1)$ and the assumptions of Theorem \ref{theo:plateau-bias} hold
\end{itemize} 
In either case
\[
T_{n,2}^{SN}(M):=
\frac{
\sqrt n\{\mathcal L_{\beta_n}(\widehat d_n)-M\}
}{
\sqrt{\widehat{\mathcal{V}}_n^{(\beta_n)}}
}
\cond
\mathcal T_{\tau_0}.
\]
In particular we have that 
\begin{align}
    \lim_n\Pr(T_{n,2}^{SN}(\Delta)>q_{1-\alpha_0}(\hat \tau_n))=\begin{cases}
        0, \quad &M<\Delta,\\
        \alpha_0, \quad &M=\Delta,\\
        1, \quad &M>\Delta.
    \end{cases}
\end{align}
\end{theo}

\begin{rem}{On boundary extrema, the choice of $\beta_n$ and the size of the quadratic remainder}
\label{rem:boundaries-choices-remainders}
    \begin{itemize}
        \item[(1)] Similar arguments can also be used to treat the case where the extrema are located at the boundary of the domain, we merely omitted them for the sake of presentational brevity. 
        \item[(2)] A version of this statement tailored towards local alternatives can be found in Appendix \ref{sec:local-alternative}. It shows that, different from the test \eqref{eq:conservative-bias-correction}, the bias-corrected test enjoys non-trivial power against any local alternative with a fixed direction.
        \item[(3)] The bias terms benefit from choosing $\beta_n$ as large as possible whereas the remainder in the Taylor expansion \eqref{eq:taylor-expansion} grows as $\beta_nn^{-1}$. For isolated extrema this indicates an optimal choice $\beta_n \simeq n^{1/3}$ whereas for plateaus the optimal choice is $\beta_n \simeq n^{\frac{1}{2+\alpha}}$. The supremum norm is usually employed when extrema are sparse (otherwise using the theoretically and computationally less demanding $L^2$ norm is a viable alternative - even when one is interested in quantifying the size of the change, as the norms are extremely similar for suitably flat functions) so that we recommend using $\beta_n=n^{1/3}$ in practice. As we will see in the simulation section the proposed method is somewhat robust with respect to the choice of $\beta_n$, too. 
        \item[(4)] The quadratic remainder in the expansion \eqref{eq:taylor-expansion} applied to $\sqrt{n}(\mathcal L_{\beta_n}(\hat d)-\mathcal L_{\beta_n}(d))$, call it $\tilde R_{n,\beta_n}$, is asymptotically negligible but may still be relevant in finite samples. The inequality
        \[
            |\tilde R_{n,\beta_n}|\leq \frac{\sqrt{n}\beta_n}{2} \sum_{i=1}^3|a_i|c_i\|\hat d_n-d\|^2_\infty\simeq \sqrt{n}\frac{17.59\beta_n}{2}\|\hat d_n-d\|^2
        \]
        gives an upper bound that, depending on the geometry of the problem, may be (close to) equality. The rightmost term can be fairly large for smaller $n$ and we therefore propose a further (asymptotically negligible) adjustment given by
        \[
            \sqrt{n}\hat r_n^2=\sqrt{n}\frac{17.59\beta_n}{2}\Big(0.1\hat \sigma\frac{\log(n)}{\sqrt{n}}\Big)^2
        \]
        where $\hat \sigma=\max_{t \in [0,1]}\hat \sigma(t)$ is the maximum of the pointwise empirical standard deviations (calculated from the residuals). The bracketed term clearly stochastically dominates $\|\hat d_n-d\|^2$ when (A1) holds, so that the resulting procedure is more conservative.  As long as $\beta_n=o(n^{1/2}\log^{-2}(n))$ this does not have an asymptotic impact while improving the level approximation for smaller sample sizes. The resulting statistic then has the form
        \[
            \frac{\sqrt{n}\mathcal (L_{\beta_n}(\hat d_n)-\Delta-\hat r_n^2)}{\sqrt{\mathcal{V}_{n}^{\beta_n}}}.
        \]
    \end{itemize}
    
\end{rem}

\section{Finite Sample Study}
In this section we will compare the size and power of the proposed test with that of \cite{dette:kokot:aue:2020}. We generate data according to the scheme
\[
    X_i(t)=\epsilon_i(t)+d(t)1\{i>n/2\},\quad i=1,...,n.
\]
i.e. $\tau_0=1/2$. Here $d(\cdot)\in C([0,1])$ determines the size and shape of the break and is given by
\begin{align}
    d_{1}(t,a)=4at(1-t), \quad d_{2}(t,a)=a\min(2t,1),
\end{align}
i.e. one setting with an isolated maximum and one with a plateau. We will test the hypotheses \eqref{eq:hypotheses-defin} for $\Delta=0.5$ so that $a\leq 0.5$ corresponds to $H_0(0.5)$ and $a>0.5$ to $H_1(0.5)$. The error process will be given by 
\[
    \epsilon_i=z_\rho(\rho\epsilon_{i-1}+\eta_i)
\]
where $\eta_i$ are iid realisations of a Brownian bridge, $\rho$ determines the strength of dependence and $z_\rho=\frac{1}{1+\rho}$ is a standardization factor that ensures that the long-run variance of $\epsilon_i(1/2)$ is $1/4$. We use a burn-in sample of size $200$ that is discarded for testing. We perform our tests at level $\alpha_0=0.05$ with
\begin{align}
    n \in \{100,200,400\}, \quad a \in \{0.25,0.5,0.75\}, \rho\in\{0,0.25,0.5,0.75\}.
\end{align}
To apply the methods from \cite{dette:kokot:aue:2020} we need to choose a bandwidth for the bootstrap and a constant $c_n$ for their extremal set estimator. We follow their choice of $c_n=0.1\log(n)$ and choose the bootstrap bandwidth with the method in \cite{rice:shang:2017}. This is a method to estimate the bandwidth for a long-run covariance estimator which is closely related to the bootstrap bandwidth. We use the therein recommended choices: pilot bandwidth $n^{1/5}$, flat-top pilot kernel and a Bartlett kernel afterwards. Test decisions are based on 500 bootstrap samples. We apply our bias-corrected test based on $T_{n,2}^{SN}(1/2)$ with $\beta_n=cn^{1/3},\ c \in \{1,2,3\}$ and the additional bias correction in Remark \ref{rem:boundaries-choices-remainders}. All empirical rejection rates we report in the following are based on 1000 samples and all involved functions are sampled on an equidistant grid with 101 points. We use the break-point estimator $\hat \tau_n$ given in \cite{dette:kokot:aue:2020} for both tests. The critical values for the selfnormalized test are calculated by simulating the respective distribution $\mathcal T_{\tau_n}$ 50000 times. \\

\begin{table}[H]
\centering
\caption{Empirical rejection probabilities for the isolated maximum design.}
\label{tab:rejection-isolated}
\begin{adjustbox}{max width=\textwidth}
\begin{tabular}{cc cccc|cccc|cccc|cccc}
\toprule
$\rho$ & $a$ & \multicolumn{4}{c}{${n=100}$} & \multicolumn{4}{c}{${n=200}$} & \multicolumn{4}{c}{${n=400}$} \\
\cmidrule(lr){3-6} \cmidrule(lr){7-10} \cmidrule(lr){11-14}
 &  & $\mathrm{SN}_{1}$ & $\mathrm{SN}_{2}$ & $\mathrm{SN}_{3}$ & \textsc{DKA} & $\mathrm{SN}_{1}$ & $\mathrm{SN}_{2}$ & $\mathrm{SN}_{3}$ & \textsc{DKA} & $\mathrm{SN}_{1}$ & $\mathrm{SN}_{2}$ & $\mathrm{SN}_{3}$ & \textsc{DKA} \\
\midrule
\multirow{3}{*}{0.00} & 0.25 & 0.000 & 0.000 & 0.000 & 0.000 & 0.000 & 0.000 & 0.000 & 0.000 & 0.000 & 0.000 & 0.000 & 0.000 \\
 & \textbf{0.50} & 0.097 & 0.074 & 0.051 & 0.112 & 0.063 & 0.043 & 0.026 & 0.102 & 0.062 & 0.040 & 0.022 & 0.079 \\
 & 0.75 & 0.806 & 0.749 & 0.681 & 0.898 & 0.941 & 0.918 & 0.887 & 0.986 & 0.998 & 0.997 & 0.993 & 1.000 \\
\midrule
\multirow{3}{*}{0.25} & 0.25 & 0.000 & 0.000 & 0.000 & 0.000 & 0.000 & 0.000 & 0.000 & 0.000 & 0.000 & 0.000 & 0.000 & 0.000 \\
 & \textbf{0.50} & 0.100 & 0.086 & 0.072 & 0.142 & 0.072 & 0.059 & 0.049 & 0.114 & 0.066 & 0.052 & 0.036 & 0.088 \\
 & 0.75 & 0.817 & 0.791 & 0.766 & 0.933 & 0.951 & 0.943 & 0.935 & 0.987 & 0.996 & 0.996 & 0.993 & 1.000 \\
\midrule
\multirow{3}{*}{0.50} & 0.25 & 0.000 & 0.000 & 0.000 & 0.000 & 0.000 & 0.000 & 0.000 & 0.000 & 0.000 & 0.000 & 0.000 & 0.000 \\
 & \textbf{0.50} & 0.098 & 0.094 & 0.084 & 0.145 & 0.077 & 0.070 & 0.056 & 0.120 & 0.071 & 0.056 & 0.042 & 0.090 \\
 & 0.75 & 0.850 & 0.832 & 0.817 & 0.930 & 0.958 & 0.952 & 0.946 & 0.996 & 0.997 & 0.997 & 0.996 & 0.999 \\
\midrule
\multirow{3}{*}{0.75} & 0.25 & 0.000 & 0.000 & 0.000 & 0.000 & 0.000 & 0.000 & 0.000 & 0.000 & 0.000 & 0.000 & 0.000 & 0.000 \\
 & \textbf{0.50} & 0.094 & 0.090 & 0.086 & 0.165 & 0.065 & 0.066 & 0.059 & 0.118 & 0.063 & 0.057 & 0.052 & 0.079 \\
 & 0.75 & 0.857 & 0.849 & 0.838 & 0.948 & 0.965 & 0.959 & 0.951 & 0.994 & 0.997 & 0.997 & 0.996 & 1.000 \\
\bottomrule
\end{tabular}
\end{adjustbox}
\vspace{2pt}
\begin{minipage}{\textwidth}
\footnotesize
\emph{Notes:} The relevance threshold is $\Delta=0.5$; hence the bold entries in the $a$ column identify the boundary of the null hypothesis. $\mathrm{SN}_{c}$ denotes the bias-corrected selfnormalized test with $\beta_n=c n^{1/3}$, and \textsc{DKA} denotes the relevant-change bootstrap of Dette, Kokot and Aue (2020).  Each entry reports the empirical rejection rate  based on 1,000 Monte Carlo replications.
\end{minipage}
\end{table}

\begin{table}[H]
\centering
\caption{Empirical rejection probabilities for the plateau design.}
\label{tab:rejection-plateau}
\begin{adjustbox}{max width=\textwidth}
\begin{tabular}{cc cccc|cccc|cccc|cccc}
\toprule
$\rho$ & $a$ & \multicolumn{4}{c}{${n=100}$} & \multicolumn{4}{c}{${n=200}$} & \multicolumn{4}{c}{${n=400}$} \\
\cmidrule(lr){3-6} \cmidrule(lr){7-10} \cmidrule(lr){11-14}
 &  & $\mathrm{SN}_{1}$ & $\mathrm{SN}_{2}$ & $\mathrm{SN}_{3}$ & \textsc{DKA} & $\mathrm{SN}_{1}$ & $\mathrm{SN}_{2}$ & $\mathrm{SN}_{3}$ & \textsc{DKA} & $\mathrm{SN}_{1}$ & $\mathrm{SN}_{2}$ & $\mathrm{SN}_{3}$ & \textsc{DKA} \\
\midrule
\multirow{3}{*}{0.00} & 0.25 & 0.000 & 0.000 & 0.000 & 0.000 & 0.000 & 0.000 & 0.000 & 0.000 & 0.000 & 0.000 & 0.000 & 0.000 \\
 & \textbf{0.50} & 0.106 & 0.081 & 0.055 & 0.164 & 0.066 & 0.056 & 0.043 & 0.139 & 0.045 & 0.029 & 0.019 & 0.095 \\
 & 0.75 & 0.981 & 0.960 & 0.935 & 1.000 & 0.999 & 0.996 & 0.989 & 1.000 & 1.000 & 1.000 & 1.000 & 1.000 \\
\midrule
\multirow{3}{*}{0.25} & 0.25 & 0.001 & 0.000 & 0.000 & 0.002 & 0.000 & 0.000 & 0.000 & 0.000 & 0.000 & 0.000 & 0.000 & 0.000 \\
 & \textbf{0.50} & 0.115 & 0.109 & 0.095 & 0.209 & 0.096 & 0.089 & 0.074 & 0.169 & 0.088 & 0.073 & 0.061 & 0.143 \\
 & 0.75 & 0.992 & 0.978 & 0.974 & 1.000 & 1.000 & 0.998 & 0.993 & 1.000 & 1.000 & 1.000 & 1.000 & 1.000 \\
\midrule
\multirow{3}{*}{0.50} & 0.25 & 0.000 & 0.000 & 0.000 & 0.000 & 0.000 & 0.000 & 0.000 & 0.000 & 0.000 & 0.000 & 0.000 & 0.000 \\
 & \textbf{0.50} & 0.138 & 0.141 & 0.139 & 0.237 & 0.103 & 0.104 & 0.102 & 0.174 & 0.075 & 0.071 & 0.066 & 0.122 \\
 & 0.75 & 0.987 & 0.970 & 0.968 & 1.000 & 0.999 & 0.996 & 0.993 & 1.000 & 1.000 & 1.000 & 0.999 & 1.000 \\
\midrule
\multirow{3}{*}{0.75} & 0.25 & 0.000 & 0.000 & 0.000 & 0.001 & 0.000 & 0.000 & 0.000 & 0.000 & 0.000 & 0.000 & 0.000 & 0.000 \\
 & \textbf{0.50} & 0.148 & 0.141 & 0.139 & 0.242 & 0.125 & 0.119 & 0.113 & 0.184 & 0.077 & 0.089 & 0.083 & 0.145 \\
 & 0.75 & 0.993 & 0.985 & 0.983 & 1.000 & 0.999 & 0.999 & 0.997 & 1.000 & 1.000 & 1.000 & 1.000 & 1.000 \\
\bottomrule
\end{tabular}
\end{adjustbox}
\vspace{2pt}
\begin{minipage}{\textwidth}
\footnotesize
\emph{Notes:} The relevance threshold is $\Delta=0.5$; hence the bold entries in the $a$ column identify the boundary of the null hypothesis. $\mathrm{SN}_{c}$ denotes the bias-corrected selfnormalized test with $\beta_n=c n^{1/3}$, and \textsc{DKA} denotes the relevant-change bootstrap of Dette, Kokot and Aue (2020).  Each entry reports the empirical rejection rate  based on 1,000 Monte Carlo replications.
\end{minipage}
\end{table}

In the interior of the null (i.e. $a=0.25$) all procedures reflect the asymptotic theory: rejection rates are essentially zero for every sample size, dependence level, and extremal geometry. This indicates that the relevant tests clearly distinguish changes that are well below the relevance threshold. At the boundary of the null (i.e. $a=0.5$) we observe substantially better size control for the selfnormalized procedures across all sample sizes, geometries and dependence levels. In fact the bootstrap-based procedure exceeds the nominal level in almost all of the considered scenarios. The plateau scenario $d_2$ seems to be substantially more difficult than the isolated scenario $d_1$ - even the selfnormalized methods need $n=400$ for somewhat accurate size control in the presence of strong dependence. All three choices of $c$ deliver comparable performances, with $c=3$ being slightly undersized for larger sample sizes and independent data. The rejection rates seem to be monotonically decreasing for larger $c$. Nonetheless it is noteworthy that the performance of the selfnormalized test is fairly robust to a wide range of possible $c$. For the alternative (i.e. $a=0.75$) we observe that the selfnormalized tests generally have less power than the bootstrap test. This is consistent with the findings in many previous studies on the power of selfnormalized tests. Nonetheless their performance is at least competitive in virtually all scenarios when taking into consideration the severe size inflation of the bootstrap test.

\section*{Acknowledgements}
The author was supported by the Aarhus University Research Foundation (AUFF), project numbers 47221 and 47388.

\subsection*{Disclosure of Generative AI Use}
The author used \textit{ChatGPT 5.6 Sol} during the preparation of this manuscript. The tool was used for 
\textit{drafting, language editing, generating R code and formatting \LaTeX}.

All statistical methods, analyses, interpretations, and 
conclusions were developed and verified by the author. Any AI-generated 
suggestions, text, or code were critically reviewed and, where necessary, 
modified by the author. The author takes full responsibility for the accuracy, integrity, and reproducibility of the work. No sensitive data were provided to the AI tool.

\section{Appendix}

\subsection{Analysis of the bias-corrected test under local alternatives}
\label{sec:local-alternative}
We first establish a general theorem about local alternatives and then derive a corollary about radial local alternatives to make transparent that our test enjoys non-trivial power against any fixed direction of deviation from $H_0(\Delta)$.\\

For the bias-corrected functional, write
\begin{align}
\label{eq:derivative-bias-corrected}
   \ell_n^{\mathcal L}
:=
D\mathcal L_{\beta_n,d_n}
=
\sum_{j=1}^3
a_jD\Psi_{c_j\beta_n,d_n}, 
\end{align}
and define the corresponding projected long-run variance by
\[
(\sigma_n^{\mathcal L})^2
:=
\E\left[
\left\{
\ell_n^{\mathcal L}(B(1))
\right\}^2
\right].
\]

\begin{theo}[Local power of the bias-corrected test]
\label{theo:bias-corrected-local}
Let \(d=d_n\in C([0,1])\) be a sequence of contrasts and put
\[
M_n:=\norminf{d_n}.
\]
Suppose that Assumptions \ref{ass:population-centering} hold with $d=d_n$. Further assume that
\[
\sup_n\|\ell_n^{\mathcal L}\|_{\Cspace^*}<\infty,
\qquad
\liminf_{n\to\infty}(\sigma_n^{\mathcal L})^2>0,\quad 0<\lim\inf_nM_n\leq \lim \sup_n M_n<\infty.
\]
Also assume that one of the following conditions holds uniformly in
\(n\):
\begin{itemize}
    \item[(1)] the isolated-extrema assumptions of
    Theorem~\ref{theo:laplace-bias} hold uniformly (i.e. the number of extrema is uniformly bounded, and there exists $\delta>0$ such that all extrema are at least $\delta$ apart from one another and at least $\delta$ away from the boundary. Moreover, the curvatures are uniformly bounded away from zero and $\sup_n \|d_n\|_{C^4}<\infty$) and
    \[
    \sqrt n\,\beta_n^{-2}=o(1);
    \]
    \item[(2)] the positive-measure-extremal-set assumptions of
    Theorem~\ref{theo:plateau-bias} hold uniformly (i.e. $m$ stays bounded away from 0, $K_n$ is uniformly bounded away from 0 and infinity and the remainder in the expansion of $F(\cdot)$ is uniform), with a common exponent
    \(\alpha>0\), and
    \[
    \sqrt n\,\beta_n^{-(1+\alpha)}=o(1).
    \]
\end{itemize}
Suppose moreover that
\[
\gamma_n
:=
\frac{\sqrt{n\tau_0(1-\tau_0)}}{\sigma_n^{\mathcal L}}
\left\{
M_n-\Delta
\right\}
\longrightarrow
\gamma\in\overline{\R}.
\]
Then
\[
\Pr\left(
T_{n,2}^{SN}(\Delta)
>
q_{1-\alpha_0}(\widehat\tau_n)
\right)
\longrightarrow
\Pr\left(
\frac{Z+\gamma}{H_{\tau_0}}
>
q_{1-\alpha_0}(\tau_0)
\right),
\]
where
\[
H_{\tau}
=
\left\{
\tau^2\int_0^1\bigl(B_1^0(u)\bigr)^2\,du
+
(1-\tau)^2\int_0^1\bigl(B_2^0(u)\bigr)^2\,du
\right\}^{1/2},
\]
and \(Z,B_1^0,B_2^0\) are mutually independent, with
\(Z\sim N(0,1)\) and \(B_1^0,B_2^0\) standard Brownian bridges.

For \(\gamma\in\R\), the limiting rejection probability is
\[
\E\left[
1-\Phi\left\{
q_{1-\alpha_0}(\tau_0)H_{\tau_0}-\gamma
\right\}
\right].
\]
If \(\gamma_n\to-\infty\), the rejection probability converges to zero,
whereas if \(\gamma_n\to+\infty\), it converges to one.
\end{theo}

\begin{proof}
The proof only requires minor modifications of the proof of
Theorem~\ref{theo:bias-corrected-sn}. By the oracle reduction for the
bias-corrected functional,
\[
\frac{
\sqrt n\{
\mathcal L_{\beta_n}(\widehat d_n)
-
\mathcal L_{\beta_n}(d_n)
\}
}{
\sqrt{\widehat{\mathcal V}_n^{(\beta_n)}}
}
=
\frac{
\ell_n^{\mathcal L}\{
\sqrt n(\widetilde d_n-d_n)
\}
}{
\sqrt{\mathcal V_n^{\mathcal L}}
}
+o_p(1),
\]
where \(\widetilde d_n\) and \(\mathcal V_n^{\mathcal L}\) denote the
corresponding oracle contrast and oracle selfnormalizer. The proof of
the projected selfnormalized invariance principle applies to the
sequence \(\ell_n^{\mathcal L}\), since its operator norms are uniformly
bounded and its projected variances are uniformly nondegenerate.
Consequently,
\[
\left(
\frac{
\sqrt{n\tau_0(1-\tau_0)}
}{
\sigma_n^{\mathcal L}
}
\ell_n^{\mathcal L}(\widetilde d_n-d_n),
\,
\frac{
\sqrt{\tau_0(1-\tau_0)}
}{
\sigma_n^{\mathcal L}
}
\sqrt{\mathcal V_n^{\mathcal L}}
\right)
\cond
(Z,H_{\tau_0}).
\]

Next we note that $\mathcal L_{\beta_n}(d_n)=M_n+o(n^{-1/2})$. This follows by the arguments given in the proofs of Theorem \ref{theo:laplace-bias} and \ref{theo:plateau-bias} which are valid in the present triangular array setting due our assumptions. Now, using this, we decompose
\[
T_{n,2}^{SN}(\Delta)
=
\frac{
\sqrt n\{
\mathcal L_{\beta_n}(\widehat d_n)
-
\mathcal L_{\beta_n}(d_n)
\}
}{
\sqrt{\widehat{\mathcal V}_n^{(\beta_n)}}
}
+
\frac{
\sqrt n\{
M_n-\Delta
\}
}{
\sqrt{\widehat{\mathcal V}_n^{(\beta_n)}}
}
+
o_{\mathbb P}(1).
\]
After multiplying numerator and denominator of the second term by
\(\sqrt{\tau_0(1-\tau_0)}/\sigma_n^{\mathcal L}\), its numerator is
\(\gamma_n\), while its denominator converges jointly to
\(H_{\tau_0}\). Slutsky's theorem therefore gives
\[
T_{n,2}^{SN}(\Delta)
\cond
\frac{Z+\gamma}{H_{\tau_0}}.
\]
Since \(\widehat\tau_n\to_p\tau_0\) and
\(q_{1-\alpha_0}(\cdot)\) is continuous, the claimed convergence of
the rejection probability follows. The cases
\(\gamma=\pm\infty\) follow from the same decomposition and the fact
that \(H_{\tau_0}>0\) almost surely.
\end{proof}

\begin{cor}[Radial local alternatives]
\label{cor:bias-corrected-radial-local}
Let \(d_0\in C([0,1])\) satisfy
\[
\norminf{d_0}=\Delta>0,
\]
and suppose that \(d_0\) satisfies either the assumptions of
Theorem~\ref{theo:laplace-bias} or those of
Theorem~\ref{theo:plateau-bias}. For fixed \(h\in\R\), define
\[
d_n
=
\left(
1+\frac{h}{\Delta\sqrt n}
\right)d_0.
\]
Assume the corresponding rate condition from
Theorem~\ref{theo:bias-corrected-local} and suppose that
\[
\sigma_n^{\mathcal L}\longrightarrow\sigma_{\mathcal L}>0.
\]
Then
\[
M_n=\norminf{d_n}
=
\Delta+\frac{h}{\sqrt n}
\]
for all sufficiently large \(n\), and
\[
\Pr\left(
T_{n,2}^{SN}(\Delta)
>
q_{1-\alpha_0}(\widehat\tau_n)
\right)
\longrightarrow
\Pr\left(
\frac{Z+\gamma_h}{H_{\tau_0}}
>
q_{1-\alpha_0}(\tau_0)
\right),
\]
where
\[
\gamma_h
=
\frac{\sqrt{\tau_0(1-\tau_0)}}{\sigma_{\mathcal L}}h.
\]
Equivalently, the limiting local power is
\[
\E\left[
1-\Phi\left\{
q_{1-\alpha_0}(\tau_0)H_{\tau_0}
-
\frac{\sqrt{\tau_0(1-\tau_0)}}{\sigma_{\mathcal L}}h
\right\}
\right].
\]
In particular, this equals \(\alpha_0\) when \(h=0\), is smaller than
\(\alpha_0\) when \(h<0\), and is larger than \(\alpha_0\) when \(h>0\).
\end{cor}

\begin{proof}
The scaling defining \(d_n\) preserves the extremal set and its
qualitative geometry. The bias expansions used in
Theorem~\ref{theo:bias-corrected-sn} therefore hold uniformly along
this sequence, with coefficients that remain bounded and converge to
their values at \(d_0\). Hence
\[
\mathcal L_{\beta_n}(d_n)
=
M_n+o(n^{-1/2})
=
\Delta+\frac{h}{\sqrt n}+o(n^{-1/2}).
\]
It follows that
\[
\gamma_n
=
\frac{\sqrt{n\tau_0(1-\tau_0)}}{\sigma_n^{\mathcal L}}
\left\{
M_n-\Delta
\right\}
\longrightarrow
\frac{\sqrt{\tau_0(1-\tau_0)}}{\sigma_{\mathcal L}}h.
\]
The result now follows directly from
Theorem~\ref{theo:bias-corrected-local}.
\end{proof}

\subsection{Proof of Theorem \ref{theo:population-centered-sn}}

For notational clarity, write
\[
D\Psi_{\beta,f}(h):=D\Psi_\beta(f)[h],
\qquad f,h\in\Cspace,
\]
and define the population-centered statistic
\[
T_{n,\Psi}^{\mathrm{SN}}
:=
T_n^{\mathrm{SN}}\bigl(\Psi_{\beta_n}(d)\bigr)
=
\frac{
\sqrt n\{\Psi_{\beta_n}(\widehat d_n)-\Psi_{\beta_n}(d)\}
}{
\sqrt{\widehat V_n^{(\beta_n)}}
}.
\]

\begin{lemma}[Differentiability and stability of the soft maximum]
\label{lem:softmax-differentiability}
Let \(\beta\neq0\). The functional
\[
\Psi_\beta(f)
=
\frac1\beta
\log\int_0^1
\left\{e^{\beta f(t)}+e^{-\beta f(t)}\right\}\,dt,
\qquad f\in\Cspace,
\]
is Fr\'echet differentiable on \(\Cspace\). More precisely, for
\(f,h\in\Cspace\),
\[
\Psi_\beta(f+h)
=
\Psi_\beta(f)+D\Psi_{\beta,f}(h)+\mathcal R_{\beta,f}(h),
\]
where
\[
D\Psi_{\beta,f}(h)
=
\frac{
\int_0^1 h(t)\{e^{\beta f(t)}-e^{-\beta f(t)}\}\,dt
}{
\int_0^1 \{e^{\beta f(t)}+e^{-\beta f(t)}\}\,dt
}.
\]
The derivative and the remainder satisfy
\[
|D\Psi_{\beta,f}(h)|\le \norminf{h},
\qquad
|\mathcal R_{\beta,f}(h)|
\le \frac{|\beta|}{2}\norminf{h}^2.
\]
In addition, the derivative is Lipschitz continuous in its base point:
for all \(f,g\in\Cspace\),
\[
\bigl\|D\Psi_{\beta,f}-D\Psi_{\beta,g}\bigr\|_{\Cspace^*}
\le |\beta|\norminf{f-g}.
\]
\end{lemma}

\begin{proof}
Define
\[
G(f)
=
\int_0^1\left\{e^{\beta f(t)}+e^{-\beta f(t)}\right\}\,dt,
\qquad
\Psi_\beta(f)=\frac1\beta\log G(f).
\]
Fix \(f,h\in\Cspace\), and put
\[
F(\theta)=\Psi_\beta(f+\theta h),
\qquad \theta \in[0,1].
\]
Taylor's formula with integral remainder gives
\[
F(1)=F(0)+F'(0)+\int_0^1(1-\theta)F''(\theta)\,d\theta.
\]
For \(g_\theta=f+\theta h\), let
\[
Z_\theta
=
\int_0^1\left\{e^{\beta g_\theta(t)}+e^{-\beta g_\theta(t)}\right\}\,dt.
\]
Differentiation yields
\[
F'(\theta)
=
\frac{
\int_0^1 h(t)\{e^{\beta g_\theta(t)}-e^{-\beta g_\theta(t)}\}\,dt
}{Z_\theta}.
\]
In particular, \(F'(0)=D\Psi_{\beta,f}(h)\), which proves the
formula for the derivative.

It is useful to introduce the probability measure \(\nu_\theta\) on
\([0,1]\times\{-1,1\}\) given by
\[
d\nu_\theta(t,a)
=
\frac{e^{\beta a g_\theta(t)}}{Z_\theta}\,dt,
\qquad a\in\{-1,1\},
\]
where counting measure is used on \(\{-1,1\}\). Then
\[
F'(\theta)=\int a h(t)\,d\nu_\theta(t,a)
\]
and a second differentiation gives
\[
F''(\theta)
=
\beta\left[
\int h(t)^2\,d\nu_\theta(t,a)
-
\left\{\int a h(t)\,d\nu_\theta(t,a)\right\}^2
\right].
\]
Consequently,
\[
|F''(\theta)|
\le |\beta|\int h(t)^2\,d\nu_\theta(t,a)
\le |\beta|\norminf{h}^2.
\]
Therefore
\[
|\mathcal R_{\beta,f}(h)|
\le
\int_0^1(1-\theta)|F''(\theta)|\,d\theta 
\le
\frac{|\beta|}{2}\norminf{h}^2.
\]
This proves the claimed second-order Taylor estimate and, in particular,
the Fr\'echet differentiability of \(\Psi_\beta\). Notice also that
\[
|D\Psi_{\beta,f}(h)|
=
\left|\int a h(t)\,d\nu_0(t,a)\right|
\le \norminf{h}.
\]

It remains to prove the stability bound. For \(f,h,k\in\Cspace\),
the second derivative is the bilinear form
\[
D^2\Psi_\beta(f)[h,k]
=
\beta\,\operatorname{cov}_{\nu_f}
\bigl(a h(t),a k(t)\bigr),
\]
where \(\nu_f\) denotes the preceding probability measure with
\(g_\theta\) replaced by \(f\). By the Cauchy--Schwarz inequality,
\[
|D^2\Psi_\beta(f)[h,k]|
\le |\beta|\norminf{h}\norminf{k}.
\]
Applying the fundamental theorem of calculus to
\(\theta \mapsto D\Psi_{\beta,g+\theta(f-g)}(h)\) gives
\begin{align*}
|D\Psi_{\beta,f}(h)-D\Psi_{\beta,g}(h)|
&\le
\int_0^1
\left|
D^2\Psi_\beta\bigl(g+\theta(f-g)\bigr)[f-g,h]
\right|\,d\theta\\
&\le
|\beta|\norminf{f-g}\norminf{h}.
\end{align*}
Taking the supremum over \(\norminf{h}\le1\) proves
\[
\bigl\|D\Psi_{\beta,f}-D\Psi_{\beta,g}\bigr\|_{\Cspace^*}
\le |\beta|\norminf{f-g}.
\]
\end{proof}

For the next lemma, define the oracle contrast
\[
\widetilde d_n
=
\frac1{k_0}\sum_{i=1}^{k_0}X_i
-
\frac1{n-k_0}\sum_{i=k_0+1}^{n}X_i,
\]
and the error bridges
\[
\widetilde B_q^L
=
\sum_{i=1}^q\epsilon_i
-
\frac{q}{k_0}\sum_{i=1}^{k_0}\epsilon_i,
\qquad 1\le q\le k_0,
\]
\[
\widetilde B_q^R
=
\sum_{i=q}^n\epsilon_i
-
\frac{n-q+1}{n-k_0}\sum_{i=k_0+1}^{n}\epsilon_i,
\qquad k_0+1\le q\le n,
\]
and the oracle projected normalizer
\[
\widetilde V_n
=
\frac1{k_0(n-k_0)}
\left[
\sum_{q=1}^{k_0}\{\ell_n(\widetilde B_q^L)\}^2
+
\sum_{q=k_0+1}^{n}\{\ell_n(\widetilde B_q^R)\}^2
\right].
\]

\begin{lemma}[Linearization and oracle reduction]
\label{lem:oracle-reduction}
Under Assumptions~\ref{ass:population-centering},
\[
\sqrt n\left[
\Psi_{\beta_n}(\widehat d_n)-\Psi_{\beta_n}(d)
-\ell_n(\widehat d_n-d)
\right]
=o_p(1),
\]
\[
\sqrt n\,\norminf{\widehat d_n-\widetilde d_n}=o_p(1)
\qquad\text{and}\qquad
\widehat V_n^{(\beta_n)}-\widetilde V_n=o_p(1).
\]
Hence
\[
T_{n,\Psi}^{\mathrm{SN}}
=
\frac{\ell_n\{\sqrt n(\widetilde d_n-d)\}}
{\sqrt{\widetilde V_n}}
+o_p(1).
\]
\end{lemma}

\begin{proof}
Assumptions~\ref{ass:wip} and \ref{ass:break-estimator} imply
\[
\norminf{\widehat d_n-d}=O_p(n^{-1/2}).
\]
The remainder estimate in
Lemma~\ref{lem:softmax-differentiability} therefore gives
\begin{align*}
&\sqrt n\left|
\Psi_{\beta_n}(\widehat d_n)-\Psi_{\beta_n}(d)
-\ell_n(\widehat d_n-d)
\right|\\
&\hspace{3em}\le
\frac{\sqrt n\,\beta_n}{2}
\norminf{\widehat d_n-d}^2
=
O_p\left(\frac{\beta_n}{\sqrt n}\right)
=o_p(1).
\end{align*}

Put \(m_n=|\widehat k_n-k_0|\). The deterministic contribution caused by
misclassified observations is of order \(m_n/n\), and therefore becomes
\(o_p(n^{-1/2})\) by Assumption~\ref{ass:break-estimator}. For the stochastic
part, tightness and asymptotic equicontinuity of the partial-sum process in
Assumption~\ref{ass:wip} imply that increments over intervals of length
\(m_n/n=o_p(1)\) are \(o_p(1)\) after multiplication by \(n^{-1/2}\).
This proves 
\[
\sqrt{n}\|\hat d_n-\widetilde d_n\|_\infty=o_p(1).
\]

It remains to compare the estimated and oracle quadratic
normalizers in more detail. We first provide some bounds on the magnitude and differences of the error bridges and their oracle versions. We then use these bounds to show that the selfnormalizer and its oracle are the same for our purposes.\\

Put
\[
m_n=|\widehat k_n-k_0|.
\]
By Assumption~\ref{ass:break-estimator},
\[
m_n=o_p(\sqrt n).
\]

The weak invariance principle implies 
\[
\sup_{r\in[0,1]}\norminf{\mathbb U_n(r)}=O_p(1),
\]
and, for every random sequence \(\delta_n\conp0\), also
\[
\omega(\mathbb U_n,\delta_n)
:=
\sup_{\substack{r,r'\in[0,1]\\ |r-r'|\le\delta_n}}
\norminf{\mathbb U_n(r)-\mathbb U_n(r')}
=o_p(1).
\]
Taking \(\delta_n=m_n/n\), it follows that partial sums over the
indices lying between \(k_0\) and \(\widehat k_n\) are
\(o_p(\sqrt n)\). Moreover, changing the coefficients
\(t/k_0\) and \(t/\widehat k_n\) in the sums defining the normalizers, or their right-hand analogues,
produces terms of order
\[
O_p\left(
n\left|\frac1{\widehat k_n}-\frac1{k_0}\right|\sqrt n
\right)
=
O_p\left(\frac{m_n}{\sqrt n}\right)
=
o_p(\sqrt n).
\]
The deterministic contribution resulting from observations assigned
to the wrong regime is bounded uniformly by a constant multiple of
\(m_n\norminf d=o_p(\sqrt n)\). Consequently,
\begin{align}
&\max_{1\le q\le \widehat k_n}
\norminf{B_q^L}
+
\max_{\widehat k_n+1\le q\le n}
\norminf{B_q^R}
\nonumber\\
&\qquad
+
\max_{1\le q\le k_0}
\norminf{\widetilde B_q^L}
+
\max_{k_0+1\le q\le n}
\norminf{\widetilde B_q^R}
=
O_p(\sqrt n),
\label{eq:uniform-bridge-order}
\end{align}
and, on the common index sets,
\begin{align}
\Delta_n
:={}&
\max_{1\le q\le \widehat k_n\wedge k_0}
\norminf{B_q^L-\widetilde B_q^L}
\nonumber\\
&\vee
\max_{(\widehat k_n\vee k_0)+1\le q\le n}
\norminf{B_q^R-\widetilde B_q^R}
=
o_p(\sqrt n).
\label{eq:uniform-bridge-difference}
\end{align}

Next, define
\[
\widehat\ell_n
=
D\Psi_{\beta_n,\widehat d_n},
\qquad
\ell_n
=
D\Psi_{\beta_n,d}.
\]
By Lemma~\ref{lem:softmax-differentiability},
\[
\eta_n
:=
\|\widehat\ell_n-\ell_n\|_{\Cspace^*}
\le
\beta_n\norminf{\widehat d_n-d}
=o_p(1).
\]
For every index belonging to the common left-hand range,
\begin{align*}
\left|
\widehat\ell_n(B_q^L)
-
\ell_n(\widetilde B_q^L)
\right|
&\le
\left|
(\widehat\ell_n-\ell_n)(B_q^L)
\right|
+
\left|
\ell_n(B_q^L-\widetilde B_q^L)
\right|
\\
&\le
\eta_n\norminf{B_q^L}
+
\norminf{B_q^L-\widetilde B_q^L}.
\end{align*}
The same bound holds for the right-hand bridges. Therefore,
by \eqref{eq:uniform-bridge-order} and
\eqref{eq:uniform-bridge-difference},
\[
\max_{\textnormal{common }q}
\left|
\widehat\ell_n(B_q^\cdot)
-
\ell_n(\widetilde B_q^\cdot)
\right|
=o_p(\sqrt n),
\]
where \(B_q^\cdot,\  \cdot\in \{L,R\}\) denotes the appropriate left or right bridge.
At the same time,
\[
\max_{\textnormal{common }q}
\left\{
\left|\widehat\ell_n(B_q^\cdot)\right|
+
\left|\ell_n(\widetilde B_q^\cdot)\right|
\right\}
=
O_p(\sqrt n),
\]
because both derivative functionals have operator norm at most one.

Introduce the unscaled quadratic sums
\begin{align*}
\widehat Q_n
={}&
\sum_{q=1}^{\widehat k_n}
\left\{\widehat\ell_n(B_q^L)\right\}^2
+
\sum_{q=\widehat k_n+1}^{n}
\left\{\widehat\ell_n(B_q^R)\right\}^2,
\\
\widetilde Q_n
={}&
\sum_{q=1}^{k_0}
\left\{\ell_n(\widetilde B_q^L)\right\}^2
+
\sum_{q=k_0+1}^{n}
\left\{\ell_n(\widetilde B_q^R)\right\}^2.
\end{align*}
For the common indices, the identity
\[
|x^2-y^2|=|x-y||x+y|
\]
gives
\begin{align*}
\sum_{\textnormal{common }q}
\left|
\left\{\widehat\ell_n(B_q^\cdot)\right\}^2
-
\left\{\ell_n(\widetilde B_q^\cdot)\right\}^2
\right|
&\le
n\,
o_p(\sqrt n)\,
O_p(\sqrt n)
\\
&=
o_p(n^2).
\end{align*}
The two quadratic sums contain at most \(m_n\) non-common terms.
By \eqref{eq:uniform-bridge-order}, every such term is \(O_p(n)\)
uniformly. Hence their total contribution is
\[
O_p(m_n n)=o_p(n^2).
\]
It follows that
\[
\widehat Q_n-\widetilde Q_n=o_p(n^2),
\qquad
\widehat Q_n=O_p(n^2).
\]

Finally, write
\[
\widehat a_n
=
\frac1{\widehat k_n(n-\widehat k_n)},
\qquad
a_{0,n}
=
\frac1{k_0(n-k_0)}.
\]
Since \(k_0/n\to\tau_0\in(0,1)\) and
\(\widehat k_n/n\conp\tau_0\),
\[
\widehat a_n=O_p(n^{-2}),
\qquad
a_{0,n}=O(n^{-2}).
\]
Furthermore,
\begin{align*}
|\widehat a_n-a_{0,n}|
&=
\frac{
\left|
k_0(n-k_0)
-
\widehat k_n(n-\widehat k_n)
\right|
}{
\widehat k_n(n-\widehat k_n)k_0(n-k_0)
}
\\
&=
\frac{
|\widehat k_n-k_0|\,
|n-\widehat k_n-k_0|
}{
\widehat k_n(n-\widehat k_n)k_0(n-k_0)
}
\\
&=
O_p\left(\frac{m_n}{n^3}\right).
\end{align*}
Because
\[
\widehat V_n^{(\beta_n)}
=
\widehat a_n\widehat Q_n,
\qquad
\widetilde V_n
=
a_{0,n}\widetilde Q_n,
\]
we obtain
\begin{align*}
\left|
\widehat V_n^{(\beta_n)}-\widetilde V_n
\right|
&\le
|\widehat a_n-a_{0,n}|\widehat Q_n
+
a_{0,n}|\widehat Q_n-\widetilde Q_n|
\\
&=
O_p\left(\frac{m_n}{n^3}\right)O_p(n^2)
+
O(n^{-2})o_p(n^2)
\\
&=
O_p\left(\frac{m_n}{n}\right)+o_p(1)
\\
&=
o_p(1).
\end{align*}
This proves
\[
\widehat V_n^{(\beta_n)}-\widetilde V_n=o_p(1).
\]

\end{proof}

\begin{lemma}[Projected selfnormalized invariance principle]
\label{lem:projected-sn}
Let \(\ell_n\in\Cspace^*\) be deterministic linear functionals satisfying
\(\sup_n\|\ell_n\|_{\Cspace^*}\le1\) and
\[
\liminf_{n\to\infty}
\E\{\ell_n(\mathbb B(1))^2\}>0.
\]
Then, under Assumption~\ref{ass:wip},
\[
\frac{\ell_n\{\sqrt n(\widetilde d_n-d)\}}
{\sqrt{\widetilde V_n}}
\cond \mathcal T_{\tau_0},
\]
where, for a standard real Brownian motion \(W\),
\[
\mathcal T_{\tau_0}
=
\frac{W(\tau_0)-\tau_0W(1)}
{\{\tau_0(1-\tau_0)\mathcal Q_{\tau_0}(W)\}^{1/2}},
\]
with
\begin{align*}
\mathcal Q_{\tau_0}(W)
={}&
\int_0^{\tau_0}
\left(W(r)-\frac r{\tau_0}W(\tau_0)\right)^2dr\\
&+
\int_{\tau_0}^1
\left(W(1)-W(r)
-\frac{1-r}{1-\tau_0}\{W(1)-W(\tau_0)\}
\right)^2dr.
\end{align*}
Equivalently,
\[
\mathcal T_{\tau_0}
\overset{d}=
\frac{Z}
{\left\{
\tau_0^2\int_0^1\mathbb B_1^0(u)^2\,du
+(1-\tau_0)^2\int_0^1\mathbb B_2^0(u)^2\,du
\right\}^{1/2}},
\]
where \(Z\sim N(0,1)\), and \(\mathbb B_1^0,\mathbb B_2^0\) are independent
standard Brownian bridges, mutually independent of \(Z\).
\end{lemma}

\begin{proof}
Because \(\Cspace\) is separable and \(\|\ell_n\|\le1\) an application of Banach-Alaoglu yields that every subsequence of $\{\ell_n\}_{n \geq 1}$
contains a further subsequence (this is where we use separability; otherwise Banach-Alaoglu does not give sequential compactness)  that converges pointwise on \(\Cspace\) to
some \(\ell\in\Cspace^*\). By the uniform boundedness principle this implies convergence
uniformly on compact subsets of \(\Cspace\). Combined with tightness of the sequential process in
Assumption~\ref{ass:wip}, this yields
\[
\ell_n\{\mathbb U_n(\cdot)\}
\cond
\ell\{\mathbb B(\cdot)\}
=\sigma W(\cdot)
\quad\text{in }C([0,1]),
\]
where \(\sigma^2=\E\{\ell(\mathbb B(1))^2\}>0\). The positivity follows
from the uniform non-degeneracy condition on the variances. By the
continuous mapping theorem, along this further subsequence,
\[
\ell_n\{\sqrt n(\widetilde d_n-d)\}
\cond
\frac{\sigma\{W(\tau_0)-\tau_0W(1)\}}
{\tau_0(1-\tau_0)},
\]
and, additionally using Riemann-sum convergence,
\begin{align*}
\widetilde V_n\cond
\frac{\sigma^2}{\tau_0(1-\tau_0)}\Bigg[&
\int_0^{\tau_0}
\left(W(r)-\frac r{\tau_0}W(\tau_0)\right)^2dr\\
&+
\int_{\tau_0}^1
\left(W(1)-W(r)
-\frac{1-r}{1-\tau_0}\{W(1)-W(\tau_0)\}
\right)^2dr
\Bigg].
\end{align*}
The limiting denominator is positive almost surely. The factor \(\sigma\)
cancels in the ratio, so every subsequence has a further subsequence with the
same limit law \(\mathcal T_{\tau_0}\). This proves convergence of the full
sequence. The Brownian-bridge representation follows by rescaling the two
segments of \(W\); their bridges are independent of the segment endpoints,
and hence independent of the standardized numerator.
\end{proof}

Finally we state the proof of Theorem \ref{theo:population-centered-sn}.
\begin{proof}
Lemma~\ref{lem:oracle-reduction} reduces the statistic to its oracle projected
version. Lemma~\ref{lem:projected-sn}, applied with
\(\ell_n=D\Psi_{\beta_n,d}\), gives the asserted weak limit. The bound
\(\|\ell_n\|_{\Cspace^*}\le1\) follows from
Lemma~\ref{lem:softmax-differentiability}, while Assumption~\ref{ass:nondegenerate}
ensures non-degeneracy. 
\end{proof}

\subsection{Proof of Theorem \ref{theo:laplace-bias}}

\begin{proof}
For \(j=1,\ldots,m\), define
\[
\phi_j(t)=s_jd(t).
\]
Then
\[
\phi_j(t_j)=M,
\qquad
\phi_j'(t_j)=0,
\qquad
\phi_j''(t_j)=-\lambda_j<0.
\]
Because the set of extremal points is finite and each \(t_j\) lies
in the interior of \([0,1]\), there exists \(\delta>0\) such that
the intervals
\[
U_j=(t_j-\delta,t_j+\delta),
\qquad j=1,\ldots,m,
\]
are pairwise disjoint and contained in \((0,1)\). By reducing
\(\delta\), if necessary, we may also assume that
\[
s_jd(t)\ge \frac M2,
\qquad t\in U_j,
\]
and that \(t_j\) is the unique maximizer of \(\phi_j\) on \(U_j\).

Since the complement
\[
K=[0,1]\setminus\bigcup_{j=1}^m U_j
\]
is compact and contains no extremal point, there exists
\(\eta>0\) such that
\[
|d(t)|\le M-\eta,
\qquad t\in K.
\]
It follows that
\[
\int_K
\left\{
e^{\beta d(t)}+e^{-\beta d(t)}
\right\}\,dt
\le
2e^{\beta(M-\eta)}.
\]
Thus the contribution from the integral over \(K\) is exponentially smaller than
\(e^{\beta M}\beta^{-1/2}\) and hence negligible for our purposes.\\

Inside \(U_j\), the term \(e^{\beta s_jd(t)}\) is the dominant
exponential, whereas the other exponential satisfies
\[
e^{-\beta s_jd(t)}
\le
e^{-\beta M/2}.
\]
Consequently,
\begin{align*}
&\int_0^1
\left\{
e^{\beta d(t)}+e^{-\beta d(t)}
\right\}\,dt
\\
&\qquad=
\sum_{j=1}^m
\int_{U_j}e^{\beta\phi_j(t)}\,dt
+
O\left(e^{\beta(M-\eta)}\right)
+
O\left(e^{-\beta M/2}\right).
\end{align*}
It therefore remains to expand
\[
J_{j,\beta}
=
\int_{U_j}e^{\beta\phi_j(t)}\,dt.
\]

For \(x\) in a neighborhood of zero, Taylor's formula gives
\[
\phi_j(t_j+x)
=
M
-\frac{\lambda_j}{2}x^2
+\frac{\gamma_j}{6}x^3
+\frac{\kappa_j}{24}x^4
+r_j(x),
\]
where
\[
r_j(x)=o(x^4)
\qquad\text{as }x\to0.
\]
Make the change of variables
\[
y=\sqrt{\beta\lambda_j}\,x.
\]
Then
\begin{align*}
J_{j,\beta}
&=
\frac{e^{\beta M}}{\sqrt{\beta\lambda_j}}
\int_{-\delta\sqrt{\beta\lambda_j}}
^{\delta\sqrt{\beta\lambda_j}}
e^{-y^2/2}
\exp\left\{
\frac{a_jy^3}{\sqrt\beta}
+
\frac{b_jy^4}{\beta}
+
\rho_{j,\beta}(y)
\right\}\,dy,
\end{align*}
where
\[
a_j=\frac{\gamma_j}{6\lambda_j^{3/2}},
\qquad
b_j=\frac{\kappa_j}{24\lambda_j^2},
\]
and
\[
\rho_{j,\beta}(y)
=
\beta
r_j\left(
\frac{y}{\sqrt{\beta\lambda_j}}
\right).
\]
For each fixed \(y\),
\[
\rho_{j,\beta}(y)
=
o(\beta^{-1}),
\]
because \(r_j(x)=o(x^4)\). For the sake of brevity we omit showing that applying this expansion within the integral is valid in detail, giving only an outline of the required steps:
\begin{itemize}
    \item[(1)] Pick $r_\beta=o(\beta^{1/6})$ and divide the integral into regions $|y|>r_\beta$ and $|y|<r_\beta$.
    \item[(2)] Use appropriate uniformity of the Taylor expansion for $|y|<r_\beta$ to justify the expansion below for the integral over that region.
    \item[(3)] Show that $\max\{\phi_j(t_j-x),\phi_j(t_j+x)\}\leq M-cx^2, c>0$ to control the contribution of the region $|y|>r_\beta$ to the integral.
\end{itemize}

Expanding the second exponential gives, for every fixed \(y\),
\begin{align*}
&\exp\left\{
\frac{a_jy^3}{\sqrt\beta}
+
\frac{b_jy^4}{\beta}
+
\rho_{j,\beta}(y)
\right\}
\\
&\qquad=
1
+
\frac{a_jy^3}{\sqrt\beta}
+
\frac1\beta
\left\{
b_jy^4+\frac{a_j^2y^6}{2}
\right\}
+
o(\beta^{-1}).
\end{align*}

Since the integration interval is symmetric,
\[
\int_{\mathbb R}y^3e^{-y^2/2}\,dy=0.
\]
Moreover,
\[
\int_{\mathbb R}e^{-y^2/2}\,dy
=
\sqrt{2\pi},
\]
\[
\int_{\mathbb R}y^4e^{-y^2/2}\,dy
=
3\sqrt{2\pi},
\]
and
\[
\int_{\mathbb R}y^6e^{-y^2/2}\,dy
=
15\sqrt{2\pi}.
\]
Therefore,
\begin{align*}
J_{j,\beta}
&=
e^{\beta M}
\sqrt{\frac{2\pi}{\beta\lambda_j}}
\left[
1+
\frac1\beta
\left\{
3b_j+\frac{15}{2}a_j^2
\right\}
+
o\left(\frac1\beta\right)
\right].
\end{align*}
Using the definitions of \(a_j\) and \(b_j\), we obtain
\[
3b_j+\frac{15}{2}a_j^2
=
\frac{\kappa_j}{8\lambda_j^2}
+
\frac{5\gamma_j^2}{24\lambda_j^3}
=
c_j.
\]
Thus
\[
J_{j,\beta}
=
e^{\beta M}
\sqrt{\frac{2\pi}{\beta}}
\lambda_j^{-1/2}
\left[
1+\frac{c_j}{\beta}
+o\left(\frac1\beta\right)
\right].
\]

Summing over the finitely many extremal points gives
\begin{align*}
&\int_0^1
\left\{
e^{\beta d(t)}+e^{-\beta d(t)}
\right\}\,dt
\\
&\qquad=
e^{\beta M}
\sqrt{\frac{2\pi}{\beta}}
\left[
S_0(d)+\frac{S_1(d)}{\beta}
+o\left(\frac1\beta\right)
\right]
\\
&\qquad=
e^{\beta M}
\sqrt{\frac{2\pi}{\beta}}
S_0(d)
\left[
1+
\frac{S_1(d)}{\beta S_0(d)}
+
o\left(\frac1\beta\right)
\right].
\end{align*}

Taking logarithms yields
\begin{align*}
&\log
\int_0^1
\left\{
e^{\beta d(t)}+e^{-\beta d(t)}
\right\}\,dt
\\
&\qquad=
\beta M
-\frac12\log\beta
+\frac12\log(2\pi)
+\log S_0(d)
\\
&\qquad\quad+
\log\left[
1+
\frac{S_1(d)}{\beta S_0(d)}
+
o\left(\frac1\beta\right)
\right].
\end{align*}
Since
\[
\log(1+x)=x+o(x)
\qquad\text{as }x\to0,
\]
the last term equals
\[
\frac{S_1(d)}{\beta S_0(d)}
+
o\left(\frac1\beta\right).
\]
Dividing by \(\beta\) proves
\[
\Psi_\beta(d)
=
M+
\frac{
-\frac12\log\beta
+\frac12\log(2\pi)
+\log S_0(d)
}{\beta}
+
\frac{S_1(d)}{S_0(d)\beta^2}
+
o(\beta^{-2}).
\]
The asserted expansion of \(b_\beta(d)\) follows immediately.
\end{proof}

\subsection{Proof of Theorem \ref{theo:plateau-bias}}

\begin{proof}
Using
\[
e^{\beta d(t)}+e^{-\beta d(t)}
=
e^{\beta|d(t)|}+e^{-\beta|d(t)|},
\]
we obtain
\begin{align*}
&e^{-\beta M}
\int_0^1
\left\{
e^{\beta d(t)}+e^{-\beta d(t)}
\right\}\,dt
\\
&\qquad=
\int_0^1 e^{-\beta g(t)}\,dt
+
e^{-\beta M}
\int_0^1 e^{-\beta|d(t)|}\,dt.
\end{align*}
The second term satisfies
\[
0
\leq
e^{-\beta M}
\int_0^1e^{-\beta|d(t)|}\,dt
\leq
e^{-\beta M},
\]
and is therefore exponentially small and negligible for our purposes.

It remains to study
\[
J_\beta
=
\int_0^1e^{-\beta g(t)}\,dt.
\]
where $g(t)=M-|d(t)|$. Let \(\mu_g\) denote the pushforward of Lebesgue measure under the
map \(g\). Then \(F\) is the distribution function of \(\mu_g\), and
\[
J_\beta
=
\int_{[0,M]}e^{-\beta u}\,dF(u).
\]
Since
\[
F(0)=m,
\]
the measure induced by \(F\) has an atom of mass \(m\) at zero.
Writing
\[
H(u)=F(u)-m,
\]
we have \(H(0)=0\), and integration by parts gives
\[
J_\beta
=
m
+
e^{-\beta M}(1-m)
+
\beta\int_0^M e^{-\beta u}H(u)\,du.
\]
By assumption,
\[
H(u)=Ku^\alpha+o(u^\alpha),
\qquad u\downarrow0.
\]

Fix \(\delta\in(0,M)\). The contribution from \([\delta,M]\) is
exponentially small, since
\[
0
\leq
\beta\int_\delta^M e^{-\beta u}H(u)\,du
\leq
\beta e^{-\beta\delta}.
\]
On \([0,\delta]\), write
\[
H(u)=Ku^\alpha+u^\alpha r(u),
\qquad
r(u)\longrightarrow0
\]
as \(u\downarrow0\). It follows that
\begin{align*}
\beta\int_0^\delta e^{-\beta u}H(u)\,du
&=
K\beta\int_0^\delta e^{-\beta u}u^\alpha\,du
+
\beta\int_0^\delta e^{-\beta u}u^\alpha r(u)\,du.
\end{align*}
After the change of variables \(v=\beta u\),
\[
K\beta\int_0^\delta e^{-\beta u}u^\alpha\,du
=
K\beta^{-\alpha}
\int_0^{\beta\delta}e^{-v}v^\alpha\,dv
=
K\Gamma(1+\alpha)\beta^{-\alpha}
+
o(\beta^{-\alpha}).
\]
Similarly,
\[
\beta\int_0^\delta e^{-\beta u}u^\alpha r(u)\,du
=
o(\beta^{-\alpha}),
\]
by dominated convergence after the same change of variables.
Consequently,
\[
J_\beta
=
m+
K\Gamma(1+\alpha)\beta^{-\alpha}
+
o(\beta^{-\alpha}).
\]
Therefore,
\[
\int_0^1
\left\{
e^{\beta d(t)}+e^{-\beta d(t)}
\right\}\,dt
=
e^{\beta M}
\left[
m+
K\Gamma(1+\alpha)\beta^{-\alpha}
+
o(\beta^{-\alpha})
\right].
\]

Taking logarithms yields
\begin{align*}
\Psi_\beta(d)
&=
M+
\frac1\beta
\log\left[
m+
K\Gamma(1+\alpha)\beta^{-\alpha}
+
o(\beta^{-\alpha})
\right]
\\
&=
M+
\frac{\log m}{\beta}
+
\frac1\beta
\log\left[
1+
\frac{K\Gamma(1+\alpha)}{m}\beta^{-\alpha}
+
o(\beta^{-\alpha})
\right].
\end{align*}
Using
\[
\log(1+x)=x+o(x),
\qquad x\to0,
\]
we obtain
\[
\Psi_\beta(d)
=
M+
\frac{\log m}{\beta}
+
\frac{K\Gamma(1+\alpha)}{m}
\beta^{-(1+\alpha)}
+
o\left(\beta^{-(1+\alpha)}\right).
\]
This proves the assertion.
\end{proof}

\subsection{Proof of Theorem \ref{theo:bias-corrected-sn}}
\begin{proof}
The proof is essentially the same as that of Theorem \ref{theo:population-centered-sn}, we will focus mainly on the necessary adaptations. Before we begin we note that the case $M=0$ needs to be handled separately, as the softmax function has derivative zero in this case. It is sufficient to show that the statistic diverges to negative infinity in this case. An application of the continuous mapping theorem and assumption (A1) yield convergence of the numerator to negative infinity at rate $\sqrt{n}$. Arguments similar to those at the end of the proof of Theorem 3.3 in \cite{Bastian:2025} yield that for any $\epsilon>0$ we may choose $b$ such that $\hat \tau_n$ takes values in the interval $[b,1-b]$ with probability $1-\epsilon$. Combining this two facts with a routine calculation yields the desired convergence to negative infinity. We omit the details and continue with the case $M>0$.

We recall the definition of $\ell^{\mathcal L}_n$ in \eqref{eq:derivative-bias-corrected} and put
\[
A_0=\sum_{i=1}^3|a_i|,
\qquad
A_1=\sum_{i=1}^3|a_i|c_i.
\]
Lemma~\ref{lem:softmax-differentiability} gives
\[
\|\ell_n^{\mathcal L}\|_{\Cspace^*}\leq A_0,
\]
as well as
\[
\left|
\mathcal L_{\beta_n}(f+h)-\mathcal L_{\beta_n}(f)
-D\mathcal L_{\beta_n,f}(h)
\right|
\leq
\frac{\beta_nA_1}{2}\norminf h^2
\]
and
\[
\left\|
D\mathcal L_{\beta_n,f}
-D\mathcal L_{\beta_n,g}
\right\|_{\Cspace^*}
\leq
\beta_nA_1\norminf{f-g}.
\]
Consequently, the proof of Lemma~\ref{lem:oracle-reduction} applies
unchanged, with constants \(A_0\) and \(A_1\) replacing \(1\). In
particular,
\[
\frac{\sqrt n\{
\mathcal L_{\beta_n}(\widehat d_n)
-\mathcal L_{\beta_n}(d)\}}
{\sqrt{\widehat{\mathcal V}_n^{(\beta_n)}}}
=
\frac{
\ell_n^{\mathcal L}\{\sqrt n(\widetilde d_n-d)\}}
{\sqrt{\widetilde{\mathcal V}_n}}
+o_p(1).
\]

It remains only to verify non-degeneracy of the associated long-run variance $(\sigma^{\mathcal L}_n)^2=\E\big[\{\ell^{\mathcal L}_n(B(1))\}^2\Big]$. Under the assumptions of
Theorem~\ref{theo:laplace-bias}, for every \(c>0\) and \(h\in\Cspace\),
\[
D\Psi_{c\beta_n,d}(h)
\longrightarrow
\frac{\sum_{j=1}^m
s_j\lambda_j^{-1/2}h(t_j)}
{S_0(d)}.
\]
Under the assumptions of Theorem~\ref{theo:plateau-bias}, the corresponding
limit is
\[
\frac1m\int_{\mathcal E(d)}
\operatorname{sgn}\{d(t)\}h(t)\,dt.
\]
Both limits are independent of \(c\). Since \(\sum_i a_i=1\),
\(\ell_n^{\mathcal L}\) and \(D\Psi_{\beta_n,d}\) therefore have the
same pointwise limit. Their uniform boundedness and
\(\E\norminf{\mathbb B(1)}^2<\infty\) imply that Assumption~\ref{ass:nondegenerate}, i.e. $\liminf_n(\sigma^{\mathcal L}_n)^2>0$, also holds for \(\ell_n^{\mathcal L}\). Lemma~\ref{lem:projected-sn},
applied after the immaterial rescaling by \(A_0\), now yields
\[
\frac{\sqrt n\{
\mathcal L_{\beta_n}(\widehat d_n)
-\mathcal L_{\beta_n}(d)\}}
{\sqrt{\widehat{\mathcal V}_n^{(\beta_n)}}}
\cond \mathcal T_{\tau_0}.
\]

Finally, the cancellation identities and
Theorems~\ref{theo:laplace-bias}--\ref{theo:plateau-bias} give
\[
\mathcal L_{\beta_n}(d)-M
=
\begin{cases}
O(\beta_n^{-2}),
&\text{under Theorem~\ref{theo:laplace-bias}},\\
O(\beta_n^{-(1+\alpha)}),
&\text{under Theorem~\ref{theo:plateau-bias}}.
\end{cases}
\]
The stated rate conditions therefore imply
\(\sqrt n\{\mathcal L_{\beta_n}(d)-M\}=o(1)\), and the asserted weak
convergence follows from Slutsky's theorem. The statements about the
rejection probabilities follow from the continuity of
\(q_{1-\alpha_0}(\cdot)\); for \(M\neq\Delta\), the deterministic term
\(\sqrt n(M-\Delta)\) dominates.
\end{proof}

\putbib
\end{bibunit}

\end{document}